\documentclass[a4paper,10pt]{article}
\usepackage[utf8]{inputenc}
\usepackage{amsmath,amsthm,amssymb}
\usepackage[all,cmtip]{xy}
\usepackage{color}
\usepackage[all,cmtip]{xy}
\usepackage{latexsym}
\usepackage{tikz}
\usetikzlibrary{matrix,arrows}
\usepackage{lmodern}

\newtheorem{theorem}{Theorem}[section]
\newtheorem{lemma}[theorem]{Lemma}
\newtheorem{proposition}[theorem]{Proposition}
\newtheorem{corollary}[theorem]{Corollary}
\newtheorem{claim}[theorem]{Claim}

\theoremstyle{definition}
\newtheorem{definition}[theorem]{Definition}

\newtheorem{definitions and remarks}[theorem]{Definitions and Remarks}

\newtheorem{examplelemma}[theorem]{Example$\setminus$Lemma}

\theoremstyle{remark}
\newtheorem{remark}[theorem]{Remark}

\numberwithin{equation}{section}

\newcommand{\xmon}[1]{\mathbf{x}^{\boldsymbol{#1}}}

\newcommand{\sm}[1]{{\scriptstyle (#1)}}

\title{Global resolution of singularities in $1$-dimensional foliated spaces}
\author{
        Andr\'e Belotto \\
        University of Toronto\\
        andre.belottodasilva@utoronto.ca
}
\date{2014}

\begin{document}
\maketitle
\section*{Abstract}
Let $M$ be an analytic manifold over $\mathbb{R}$ or $\mathbb{C}$, $\theta$ a $1$-dimensional Log-Canonical (resp. monomial) singular distribution and $\mathcal{I}$ a coherent ideal sheaf defined on $M$. We prove the existence of a resolution of singularities for $\mathcal{I}$ that preserves the Log-Canonicity (resp. monomiality) of the singularities of $\theta$. Furthermore, we apply this result to provide a resolution of a family of ideal sheaves when the dimension of the parameter space is equal to the dimension of the ambient space minus one.
\def\contentsname{Contents}
\tableofcontents 
\thispagestyle{empty}
\newpage
\section{Introduction}
A \textit{foliated ideal sheaf} is a quadruple $(M,\theta, \mathcal{I},E)$, where: $M$ is a smooth analytic manifold of dimension $n$ over a field $\mathbb{K}$ (where $\mathbb{K}$ is $\mathbb{R}$ or $\mathbb{C}$); $\mathcal{I}$ is a coherent and everywhere non-zero ideal sheaf of $M$; $E$ is an ordered collection $E = (E^{(1)},...,E^{(l)})$, where each $E^{(i)}$ is a smooth divisor on $M$ such that $\sum_i E^{(i)}$ is a reduced divisor with simple normal crossings; $\theta$ is an involutive singular distribution defined over $M$ and everywhere tangent to $E$. In the same notation, a \textit{foliated analytic manifold} is the triple $(M,\theta,E)$ and an \textit{idealistic ideal} is the triple $(M,\mathcal{I},E)$.\\
\\
The main objective of this work is to find a resolution of singularities for $\mathcal{I}$ that preserves the class of singularities of $\theta$. In order to be precise and set notation we briefly recall some basic notions of singular distributions and resolution of singularities:
\begin{itemize}
\item \textbf{Singular distributions} (we follow \cite{BaumBott}): Let $Der_M$ denote the sheaf of analytic vector fields over $M$, i.e.,the sheaf of analytic sections of $TM$. An {\em involutive singular distribution} is a coherent sub-sheaf $\theta$ of $Der_M$ such that for each point $p$ in $M$ the stalk $\theta_p:=\theta . \mathcal{O}_p$ is closed under the Lie bracket operation.\\
\\
Consider the quotient sheaf $Q = Der_M/ \theta$.  The {\em singular set} of $\theta$ is defined by the closed analytic subset $S(\theta) = \{p \in M : Q_p \text{ is not a free $\mathcal{O}_p$ module}\}$. A singular distribution $\theta$ is called \textit{regular} if $S(\theta)=\emptyset$. On $M \setminus S(\theta)$ there exists an unique analytic subbundle $L$ of $TM |_{ M\setminus S(\theta)}$ such that $\theta$ is the sheaf of analytic sections of $L$. We assume that the dimension of the $\mathbb{K}$ vector space $L_p$ is the same for all points $p$ in $M \setminus S$ (this always holds if $M$ is connected). This dimension is called the {\em leaf dimension} of $\theta$ and is denoted by $d$. In this case $\theta$ is called an involutive \textit{$d$-singular distribution}.\\
\\
A blowing-up $\sigma: (\widetilde{M},\widetilde{E})\to (M,E)$ is \textit{admissible} if the center $\mathcal{C}$ is a closed and regular sub-manifold of $M$ that has simple normal crossings with $E$ (see pages 137-138 of \cite{Ko} for details). Finally, given an admissible blowing-up:
\[
 \sigma:(\widetilde{M},\widetilde{\theta},\widetilde{E}) \to (M,\theta,E)
\]
we define the singular distribution $\widetilde{\theta}$ as the strict transform of $\theta$ intersected with $Der_{\widetilde{M}}(-\widetilde{E})$.
\item \textbf{Resolution of singularities:}(we follow \cite{Ko}) The \textit{support} of the ideal sheaf $\mathcal{I}$ is the subset $V(\mathcal{I}) := \{p \in M ; \mathcal{I}.\mathcal{O}_p \subset m_p\}$ where $m_p$ is the maximal ideal of the structural ring $\mathcal{O}_p$.\\
\\
Given an admissible blowing-up $\sigma: (\widetilde{M},\widetilde{E})\to (M,E)$ we say that it has \textit{order one} in respect to $\mathcal{I}$ (or to $(M,\theta,\mathcal{I},E)$) if the center $\mathcal{C}$ is contained in the support of $\mathcal{I}$ (see definition 3.65 of \cite{Ko} for details). In this case: the \textit{total transform} of the ideal sheaf $\mathcal{I}$ is the ideal sheaf $\mathcal{I}^{\ast}:=\mathcal{I} .\mathcal{O}_{\widetilde{M}}$; the \textit{controlled transform} of the ideal sheaf $\mathcal{I}$ is the ideal sheaf $\mathcal{I}^c:= \mathcal{O}(F) \mathcal{I}^{\ast}$, where $F$ stands for the exceptional divisor of the blowing-up (see subsection 3.58 of \cite{Ko}). An admissible blowing-up of order one of a foliated ideal sheaf $(M,\theta, \mathcal{I},E)$ is the mapping:
\[
\sigma: (\widetilde{M},\widetilde{\theta}, \widetilde{\mathcal{I}},\widetilde{E}) \longrightarrow (M,\theta,\mathcal{I},E)
\]
where the ideal sheaf $\widetilde{\mathcal{I}}$ is the controlled transform of $\mathcal{I}$. We extend this notion to a sequence of blowings-up in the obvious manner, i.e., a sequence of \text{admissible blowings-up of order one} is a sequence:
\[
 \begin{tikzpicture}
  \matrix (m) [matrix of math nodes,row sep=3em,column sep=3em,minimum width=1em]
  {(M_r,\theta_r,\mathcal{I}_r,E_r) & \cdots & (M_0,\theta_0,\mathcal{I}_0,E_0)\\};
  \path[-stealth]
    (m-1-1) edge node [above] {$\sigma_r$} (m-1-2)
    (m-1-2) edge node [above] {$\sigma_1$} (m-1-3);
\end{tikzpicture}
\]
where each blowing-up $\sigma_{i+1}$ is admissible of order one in respect to $(M_i,\theta_i, \allowbreak \mathcal{I}_i,E_i)$. In particular, in this case the composition $\boldsymbol{\sigma} = \sigma_1 \circ \dots \circ \sigma_r$ is an isomorphism over $M \setminus V(\mathcal{I})$. A \textit{resolution} of an ideal sheaf $\mathcal{I}$ is a sequence of admissible blowings-up of order one
such that $\mathcal{I}_r = \mathcal{O}_{M_r}$. In particular, $\mathcal{I} . \mathcal{O}_{M_r}$ is the ideal sheaf of a SNC divisor on $M_r$ with support contained in $E_r$. We remark that the existence of a resolution of $\mathcal{I}$ is first proved by Hironaka in \cite{Hir}, and more modern proofs can be found e.g. \cite{BM,Ko,Vil4,Wo2}.
\end{itemize}
 Now, we start our study by defining which class of singularities of $\theta$ we want to preserve. For example, if we assume that $\theta$ is a regular singular distribution, we could try to resolve the singularities of $\mathcal{I}$ in such a way that the singular distribution $\theta_r$ is also regular. Unfortunately, it is easy to get examples of foliated ideal sheaves whose resolution necessarily breaks the regularity of a distribution:\\
\\
\textbf{Example: \label{Ex:BrRe}} If $(M,\theta,\mathcal{I},E)=(\mathbb{C}^2,\frac{\partial}{\partial x},(x,y),\emptyset)$ then the only possible strategy for a resolution is to blow-up the origin, which breaks the regularity of the distribution.\\
\\
Since the class of regular foliations is too restrictive for our purposes, we work with two classes of singular foliations which will be stable under suitably chosen blowings-up. More precisely:
\begin{itemize}
 \item Log-Canonical foliations where introduced by Mcquillan (see \cite{Mc}) and correspond to the class of minimal singularities of general singular foliations, i.e.,it is the best kind of foliations one can expect to obtain from reduction of singularities Theorems (see details in section $\ref{ssec:LogCanonical}$);
 \item $\mathcal{R}$-monomial foliations are defined in section \ref{ssec:Mon} and are deeply related with the notion of resolution in families and monomialization of maps (their leaves correspond to the level curves of a monomial map - see Example$\setminus$Lemma \ref{exl:Fi}). Furthermore, it is reasonable to speculate that they are minimal foliations (in the above sense) of totally integrable singular foliations.
\end{itemize}
Our main result provides a global resolution of singularities that preserves the two above classes under the hypothesis that $\theta$ has leaf dimension one. In order to be precise, since we are working in the analytic category, we define the notion of \textit{local foliated ideal sheaves} as quintuples $(M,M_0,\theta,\mathcal{I},E)$ where $M_0$ is an open relatively compact subset of $M$. We can now present the main Theorem of this work:
%
%
%
\begin{theorem}
Let $(M,M_0,\theta,\mathcal{I},E)$ be a local foliated ideal sheaf and suppose that $\theta$ is Log-Canonical (resp. $\mathcal{R}$-monomial) and has leaf dimension equal to $1$, then there exists a sequence of admissible blowings-up of order one:
\[
\begin{tikzpicture}
  \matrix (m) [matrix of math nodes,row sep=3em,column sep=3em,minimum width=2em]
  {
     (M_r,\theta_r,\mathcal{I}_r,E_r) & \cdots & (M_1,\theta_1,\mathcal{I}_1,E_1) & (M_0,\theta_0,\mathcal{I}_0,E_0)\\};
  \path[-stealth]
    (m-1-1) edge node [above] {$\sigma_r$} (m-1-2)
    (m-1-2) edge node [above] {$\sigma_2$} (m-1-3)
    (m-1-3) edge node [above] {$\sigma_1$} (m-1-4);
\end{tikzpicture}
\]
such that $\mathcal{I}_r = \mathcal{O}_{M_r}$ (i.e.,a resolution of $\mathcal{I}$) and $\theta_r$ is Log-Canonical (resp. $\mathcal{R}$-monomial).
\label{th:Hi1simpl}
\end{theorem}
\begin{remark}
This Theorem is a Corollary of Theorem  $\ref{th:Hi1}$ where we also prove that this resolution is functorial in respect to a certain class of smooth morphisms called \textit{chain-preserving} smooth morphisms (see section \ref{sec:SMCPM} for a precise definition).
\end{remark}
\begin{remark}
The class of $\mathbb{Z}$-monomial foliations is possibly the smallest class of foliations where we can resolve singularities of $\mathcal{I}$ preserving the class of the foliation (at least for leaf-dimension $1$).
\end{remark}
%
The originality of this result comes from the fact that the searched resolution can not be given by the usual Hironaka's algorithm (this is exemplified in section \ref{sec:Ex}). The proof relies in an invariant that we call tg-order (abbreviation for tangency order - see section \ref{sec:CIS} for the precise definition). This invariant measures the order of tangency between an ideal sheaf $\mathcal{I}$ and a singular distribution $\theta$, even if the objects are singular. In particular, if $\theta$ is equal to $Der_M$, the order of tangency coincides with the usual multiplicity of the ideal sheaf. We then apply Hironaka's algorithm to the maximal tangency-order locus. This guarantees some necessary ``compatibility" conditions between each blowing-up and the transforms of the singular distribution. These ``compatibility" conditions are formalized by the notion of $\theta$\textit{-admissible} centers (see section \ref{sec:Tadm} for more details).
\begin{remark}
Although all proofs and results of this manuscript are set in the analytic category, they are also valid for the algebraic category if one consider etale neighborhoods instead of analytic neighborhoods.
\end{remark}
\subsection{Example}
\label{sec:Ex}
Let us start with a simple example that illustrates the additional difficulties appearing in the problem under study. We consider the Log-Canonical foliated ideal sheaf $(M,\theta,\mathcal{I},E) = (\mathbb{C}^3,\theta,\mathcal{I},\emptyset)$, where $\theta$ is generated by the regular vector-field $\partial=\frac{\partial}{\partial z} + z \frac{\partial}{\partial x}$ and $\mathcal{I}$ is the ideal generated by $(x,y)$. Notice that the admissible blowing-up of order one $\sigma: (\widetilde{M},\widetilde{\theta},\widetilde{\mathcal{I}},\widetilde{E}) \longrightarrow (M,\theta,\mathcal{I},E)$ with center $\mathcal{C} = V(x,y)$ provides a resolution of $\mathcal{I}$. Nevertheless, in the $x$-chart ($x =\widetilde{x}$, $y=\widetilde{y}\widetilde{x}$ and $z=\widetilde{z}\widetilde{x}$), the singular distribution $\widetilde{\theta}$ is generated by:
\[
\widetilde{\partial} = \widetilde{x} \frac{\partial}{\partial \widetilde{z}} + \widetilde{z}( \widetilde{x}\frac{\partial}{\partial \widetilde{x}}- \widetilde{y}\frac{\partial}{\partial \widetilde{y}})
\]
which is not Log-Canonical since the linear part is nilpotent. Thus, in order to preserve the Log-Canonicity, we are forced to blow-up the origin first: $\sigma_1: (M_1,\theta_1,\mathcal{I}_1,E_1) \longrightarrow (M,\theta,\mathcal{I},E)$. In this case, the interesting chart is the $z$-chart ($x =x_1z_1$, $y=y_1z_1$ and $z=z_1$), where we obtain:
\[
\begin{array}{cc}
\mathcal{I}_1=(x_1,y_1) & \partial_1 = z_1 \frac{\partial}{\partial z_1}+(z_1-x_1) \frac{\partial}{\partial x_1}-y_1 \frac{\partial}{\partial y_1}
\end{array}
\]
and the vector-field $\partial_1$ is clearly Log-Canonical. Now, once again a blowing-up with center $(x_1,y_1)$ would break the Log-Canonicity of the singular distribution, so we are forced once again to consider the blowing-up of the origin: $\sigma_2: (M_2,\theta_2,\mathcal{I}_2,E_2) \longrightarrow  (M_1,\theta_1,\mathcal{I}_1,E_1)$. In this case the interesting chart is the $z_1$-chart ($x_1 =x_2z_2$, $y_1=y_1z_1$ and $z_1=z_2$), where we obtain:
\[
\begin{array}{cc}
\mathcal{I}_2=(x_2,y_2) & \partial_2= z_2 \frac{\partial}{\partial z_2} +(1-2x_2) \frac{\partial}{\partial x_2}-2y_2 \frac{\partial}{\partial y_2}
\end{array}
\]
and the vector-field $\partial_2$ is clearly Log-Canonical. Now the reader can easily verify that the blowing-up with center $(x_2,y_2)$ resolves $\mathcal{I}$ and preserve the Log-Canonicity of the singular distribution.
\subsection{Applications and Related Problems}
\label{sec:Apl}
In this section we indicate two research topics that are related to this work. Other related subjects are treated in \cite{BeloT}.  For instance, in chapter 8 of \cite{BeloT}, the main Theorem \ref{th:Hi1simpl} is used to study the existence of slices for certain lie group actions over an analytic and regular variety.\\
\\
\textbf{Application: Resolution in families.}\\
\\
A \textit{smooth family of ideal sheaves} is given by a quadruple $(B,\Lambda, \pi, \mathcal{I})$ where: the ambient space $B$ and the parameter space $\Lambda$ are two smooth analytic manifolds; the morphism $\pi: B \longrightarrow \Lambda$ is smooth; The ideal sheaf $\mathcal{I}$ is coherent and everywhere non-zero over $B$. In the context of this work, a smooth family of ideal sheaves $(B,\Lambda, \pi, \mathcal{I})$ gives rise to a foliated ideal sheaf $(B,\theta,\mathcal{I},\emptyset)$, where $\theta$ is the maximal regular distribution such that $(D \pi) \theta = 0$.\\
\\
Many works have addressed resolution processes for families of ideal sheaves. By this, we intuitively mean a resolution of $\mathcal{I}$ that `preserves' the structure of family. The precise meaning of resolution in families is not well-established in the literature, e.g., two non-equivalent definitions are proposed in \cite{Vil1,Vil2}. In \cite{Vil2} the working definition is based on the notion of quasi-smooth morphism, which is deeply related to the notion of codimension one monomial singular distributions.\\
\\
Based on this observation we propose the following definition  for  resolution in families:\\
\\
\textbf{Uniform Resolution in Families of Ideal sheaves:} An \textit{uniform resolution} of a smooth family of ideal sheaves $(B,\Lambda, \pi, \mathcal{I})$ is an admissible sequence of blowings-up of order one:
\[
\begin{tikzpicture}
  \matrix (m) [matrix of math nodes,row sep=3em,column sep=3em,minimum width=2em]
  {(B_r,\theta_r,\mathcal{I}_r,E_r) & \cdots & (B_1,\theta_1,\mathcal{I}_1,E_1) & (B,\theta,\mathcal{I},\emptyset)\\};
  \path[-stealth]
    (m-1-1) edge node [above] {$\sigma_r$} (m-1-2)
    (m-1-2) edge node [above] {$\sigma_2$} (m-1-3)
    (m-1-3) edge node [above] {$\sigma_1$} (m-1-4);
\end{tikzpicture}
\]
such that $\mathcal{I}_r= \mathcal{O}_{B_r}$ and $\theta_r$ is $\mathbb{Z}$-monomial.\\
\\
We remark that this notion of resolution in families has originally been proposed at $\cite{RZ}$ in the context of smooth families of planar foliations by curves. In particular, it is an essential step in Roussarie's program for the existential part of the Hilbert $16^{th}$ Problem.\\
\\
As a consequence of the results in $\cite{Vil2}$, it is already proved the existence of an uniform resolution in families for the case where $dim \Lambda =1$ (under the additional hypotheses that the morphism $\pi$ is flat over $V(\mathcal{I})$). Furthermore, as a consequence of Theorem $\ref{th:Hi1simpl}$, the following result is immediate:
\begin{corollary}
Given a smooth family of ideal sheaves $(B,\Lambda, \pi, \mathcal{I})$ such that $dim \Lambda = dim B -1$ then, for every relatively compact open set $B_0 \subset B$, there exists an uniform resolution of $(B_0, \Lambda,\pi_0,\mathcal{I}_0)$ where $\pi_0= \pi|_{B_0}$ and $\mathcal{I}_0 = \mathcal{I} . \mathcal{O}_{B_0}$.
\end{corollary}
\begin{flushleft}
\textbf{Related Problem: Monomialization of maps.} 
\end{flushleft}
An analytic map $\Phi: M \longrightarrow N$ is \textit{monomial} if at every point $p$ in $M$, there exists a coordinate system $\mathbf{x}=(x_1,...,x_m)$ over $\mathcal{O}_p$ and $\mathbf{y}=(y_1,...,y_n)$ of $\mathcal{O}_{\Phi(p)}$ such that:
\[
\Phi(x) = (\Phi_1(x),...,\Phi_n(x)) = \left(\mathbf{x}^{\alpha_1}, ... ,\mathbf{x}^{\alpha_n}\right)
\]
where the multi-indexes $\alpha_i = (\alpha_{i,1}, \dots, \alpha_{i,m})$ are all linearly independent, i.e.,$\alpha_1 \wedge \dots \wedge \alpha_n \not\equiv0$. We now consider the following problem (for a precise formulation see $\cite{King,Cut}$): given an analytic map $\Phi: M \longrightarrow N$ such that $d\Phi$ is generically of maximal rank, can we assume, up to a sequence of blowings-up in $M$ and $N$, that the map $\Phi: M \longrightarrow N$ is monomial?\\
\\
The best results, to date, on the algebraic category are given in a series of articles of Cutkosky $\cite{Cut,Cut2,Cut3}$ (where local uniformization for any dimension and global monomialization for maps from three folds to surfaces is proved) and an article of Dan Abramovich, Jan Denef and Kalle Karu $\cite{Abr}$ (where monomialization by modifications is proved). Nevertheless, the problem in all its generality is still not solved and very few results seem to exist for the analytic category.\\
\\
The present work naturally relates with this problem since monomial singular distributions are level curves of monomial maps. We shall explore such a relation in a forthcoming manuscript (which can already be found in arXiv, see \cite{Bel2}) where we study a local monomialization algorithm of first integrals. 
%
%
%
\section{Singular Distributions}
\subsection{Log-Canonical Singular Distributions}
\label{ssec:LogCanonical}
For planar vector-fields, one can use the Theorem of reduction of singularities of Seidenberg to work only with reduced singularities. In general, one can only expect Theorems of reduction of singularities that reduced the problem to the class od \textit{canonical singularities} as introduced by McQuillan (see \cite{Mc}), following the approach of the Mori program.\\
\\
In order to well-define the notion of Log-Canonical singularities, let $M_s$ be the set $M \setminus S(\theta)$ and $j: (M_s,\theta_s) \rightarrow (M,\theta)$ be the injection from $M_s $ to $M$. Since $\theta_s$ is regular over $M_s$, we can define the canonical sheaf $K_{\theta_s} := det(\theta^{\vee})$ over $M_s$, where $\theta^{\vee}$ stands for the dual of $\theta$. We now define $K_{\theta}:= j_{\ast}K_{\theta_s}$.\\
\\
Now, let $\pi : (\widetilde{M},\widetilde{\theta},\widetilde{E}) \to (M,\theta,E)$ be any sequence of blowings-up. We can write:
\[
 K_{\widetilde{\theta}} = \pi^{\ast}K_{\theta} + \sum a(F,M,\theta) F
\]
where $a(F,M,\theta)$ is independent of the morphism $\pi$ and depends only on the discrete valuation that defines $F$ (which is a reduced exceptional divisor). It is called the discrepancy of $(M,\theta)$ at $F$. 
Now, let:
\[
 \epsilon(F) = \left\{ \begin{array}{l}
                       0 \text{ if $F$ is invariant by the singular distribution}\\
                       1 \text{ if $F$ is not invariant by the singular distribution}
                      \end{array}
 \right.
\]
We are now ready to define the notion of Log-Canonical Singularities:
\begin{definition}
The foliated manifold $(M,\theta,E)$ is said to be log-canonical if $a(\widetilde{E}_j,M,\theta) \geq -\epsilon(\widetilde{E}_j) $ for all sequence of blowings-up $\pi:(\widetilde{M},\widetilde{\theta},\widetilde{E})\to (M,\theta,E)$.
\end{definition}
In particular, if the singular distribution has leaf dimension equals to one, the Log-Canonical class coincides with the notion of elementary singularities. More precisely, given a vector-field $\partial$ in $Der_M .\mathcal{O}_p$, one of the following possibilities must be true:
\begin{itemize}
 \item[i]) $\partial$ is regular at $p$, i.e.,there exists a function $f$ in $\mathcal{O}_p$ such that $\partial(f)$ is a unit;
 \item[ii]) Otherwise, $\partial$ is singular at $p$, i.e.,the vector-field $\partial$ leaves the maximal ideal $\mathbf{m}_p$ invariant and gives rise to an endomorphism between the Zariski tangent spaces:
\[
  \bar{\partial} :  \frac{\boldsymbol{m}_p}{\boldsymbol{m}^2_p} \to  \frac{\boldsymbol{m}_p}{\boldsymbol{m}^2_p}
  \]
\end{itemize}
\begin{lemma}(see Fact I.ii.4 in \cite{McP}) A $1$-singular distribution $\theta$ is \textit{log-canonical} if, and only if, at each point $p$ in $M$, the singular distribution $\theta.\mathcal{O}_p$ is generated by a vector-field $\partial_p$ which is either in case $[i]$ (i.e.,it is regular) or it is in case $[ii]$ and $\bar{\partial}$ is a non-nilpotent endomorphism.
\end{lemma}
%
%
%
%
\subsection{Monomial Singular Distribution}
\label{ssec:Mon}
The class of Monomial Singular Distributions is a sub-class of the Log-Canonical Singular Distributions. Its motivation lies in the study of families of ideals or vector-fields and in the study of monomialization of maps. We come back to this discussion after we give its definition:
\begin{definition}
\label{def:mon}
Given a foliated manifold $(M,\theta,E)$, we say that the singular distribution $\theta$ is $R$-\textit{monomial} at a point $p$ if there exists set of generators $\{\partial_1,...,\partial_d\}$ of $\theta . \mathcal{O}_p$ and a coordinate system $\mathbf{x} = (x_1,\ldots,  x_m)$ centered at $p$ such that:
\begin{itemize}
\item[i)] The exceptional divisor $E$ is locally equal to $\{x_1 \cdots x_l=0\}$ for some $l\leq m$;
\item[ii)] The singular distribution $\theta$ is everywhere tangent to $E$, i.e.,$\theta \subset Der_M(-E)$;
\item[iii)] Apart from re-indexing, the vector-fields are of the form:
\[
 \begin{aligned}
  \partial_i &=  \partial x_{i} \text{ or}\\
  \partial_i &= \sum^{m}_{j=1} \alpha_{i,j}x_j \partial x_j \text{ (where }\alpha_{i,j} \in R\text{)}
 \end{aligned}
\]
\item[iv)] If $\omega \subset Der_M(-logE)$ is a $d$-singular distribution such that $\theta \subset \omega$ then $\theta = \omega$.
%
\end{itemize}
In this case, we say that $\mathbf{x}$ is a \textit{monomial coordinate system}, and that $\{\partial_1,\dots,\partial_d\}$ is a \textit{monomial basis}.
\end{definition}
%
The Example$\setminus$Lemma below shows an important feature of $R$-monomial singular distribution:
%
\begin{examplelemma}[Monomial First Integrals]
\label{exl:Fi}
A singular distribution $\theta$ is $R$-monomial if, and only if, for any monomial coordinate system $\mathbf{x} = (x_1,\ldots,  x_m)$ centered at $p$ there exists $m-d$ monomials $(\xmon{\beta_1},\dots,\xmon{\beta_{m-d}})$, where $\boldsymbol{\beta_i} \in R^m$ for all $i$ and $\boldsymbol{\beta_1}\wedge \dots \wedge \boldsymbol{\beta_{m-d}} \neq 0$, such that
 \[
\theta.\mathcal{O}_p =\{\partial \in Der_p(-E);\text{ } \partial(\xmon{\beta_i})\equiv 0 \text{ for all } i\}
\]
\end{examplelemma}
\begin{remark}
The proof of Example$\setminus$Lemma \ref{exl:Fi} can be find in \cite{BeloT}, Lemma $2.2.2$. We do not reproduce its proof because the result is not going to be explicitly used in this manuscript.
\end{remark}
This result shows in more detail the relation between monomial singular distributions and monomialization of maps (see \cite{Cut2} for results in monomialization of maps). Indeed, if a holomorphic map is monomial, the foliation generated by its level curves is $\mathbb{Z}$-monomial. Furthermore, in the study of families, the notion of quasi-smoothness (see \cite{Vil2} for a definition) is closely related to a $\mathbb{Z}$-monomial foliation. We now turn to some important properties of $R$-monomial singular distributions:
\begin{lemma}
If $\theta$ is a $1$-singular distribution which is $R$-monomial at a point $p$ in $M$, then there exists an open neighborhood $U$ of $p$ such that $\theta$ is $R$-monomial at every point $q$ in $U$.
\label{lem:Gmonloc}
\end{lemma}
\begin{remark}
 The above Lemma can be enunciated for any $d$-singular distribution $R$-monomial (see Lemma 2.2.1 of \cite{BeloT}). In this manuscript, we only prove for $1$-singular distribution because we only need this level of generality.
\end{remark}
\begin{proof}
Let $\mathbf{x}$ be a monomial coordinate system defined in an open neighborhood $U$ and $\partial$ a vector-field that generates $\theta.\mathcal{O}_U$. Since $\theta$ is $R$-monomial, either $\partial$ is a regular vector-field and the result is obvious, or
\[
  \partial= \sum^{m}_{j=1} \alpha_{j}x_j \partial x_j \text{ (where }\alpha_{j} \in R\text{)}
\]
Now, let $q = (q_1,\dots, q_m)$ be another point of $U$, and let $\mathbf{y} = (y_1, \dots ,y_m) = (x_1-q_1,\dots, x_m-q_m)$ be a coordinate system centered at $q$. Then:
\[
  \partial . \mathcal{O}_q= \sum^{m}_{j=1} \alpha_{j}(y_j + q_j) \partial y_j
\]
and either this vector-field is regular, or $q_j \alpha_j = 0$ for all $j$. In both cases, it is clear that the singular distribution is monomial at $q$.
\end{proof}
\begin{flushleft}
We end this subsection proving another important property of $R$-monomial singularities:
\end{flushleft}
\begin{lemma}
 Suppose that $\theta$ is a $1$-singular distribution $R$-monomial at a point $p$ in $M$ and that $\mathcal{I}$ is an ideal sheaf $\theta$-invariant, i.e.,such that $\theta[\mathcal{I}]\subset \mathcal{I}$. Let $\partial$ be a monomial vector-field that generates $\theta$ at $p$. Then, there exists a system of generators $\{f_1,\dots,f_n\}$ of $\mathcal{I}.\mathcal{O}_p$ such that $\partial f_i = K_i f_i$, where $K_i$ is a constant in $R$.
%
\label{lem:InvNSing}
\end{lemma}
\begin{remark}
 In particular, this shows that condition $[i]$ of the Definition of monomial singularities $\ref{def:mon}$ can be deduced from the other three properties $[ii-iv]$.
\end{remark}
\begin{proof}
Let us fix a monomial coordinate system $\mathbf{x} = (x_1,\dots,x_m)$ and a system of generators $(f_1, \dots, f_n)$ of $\mathcal{I}.\mathcal{O}_p$. Let us first assume that $\partial$ is a singular vector-field and, thus, $\partial = \sum_i \alpha_i x_i \partial_{x_i}$. So, given any monomial $\xmon{\beta}$, we have that
\[
\partial(\xmon{\beta}) = \sum^m_{i=1} \alpha_i \beta_i \xmon{\beta} = K_{\boldsymbol{\beta}} \xmon{\beta}
\]
For some $K_{\boldsymbol{\beta}} \in R$. Since the number of different monomials in a Taylor expansion is countable, there exists a countable set $R^{'} \subset R$ such that $K_{\boldsymbol{\beta}} \in R^{'}$ for all $\boldsymbol{\beta} \in \mathbb{Z}^{m}$. This allow us to rewrite the Taylor expansion of each generator $f_i$ in the following form:
\[
f_i(\mathbf{x})= \sum_{j \in \mathbb{N}} h_{i,j}(\mathbf{x})
\]
where $\partial(h_{i,j})= K_j h_{i,j}$ with $K_j \in R^{'}$ and $K_j \neq K_k$ whenever $j\neq k$. Furthermore, let us notice that $h_{i,j}(\mathbf{x})$ are convergent Taylor series (because $f_i$ is absolutely convergent in a neighborhood), which implies that $h_{i,j}(\mathbf{x}) \in \mathcal{O}_p$. We claim that all functions $h_{i,j}$ are contained in $\mathcal{I}.\mathcal{O}_p$. Indeed, let us show that $h_{1,0}$ is in the ideal (the proof of the other functions follows an analogous argument). Let $g_0 = f_1$ and
\[
\begin{aligned}
 g_1 &= \frac{1}{K_0-K_1}(\partial(f_1) - K_1 f_1 ) = \frac{1}{K_0-K_1}[ \sum_{j \in \mathbb{N}} K_j h_{1,j}(x) - K_1 \sum_{j \in \mathbb{N}} h_{1,j}(x) ] \\
 &= h_{1,0} + \sum_{j \geq 2} \gamma_{1,j}\sm{1} h_{1,j} \in \mathcal{I}.\mathcal{O}_p
\end{aligned}
\]
where $\gamma_{1,j}\sm{1} = \frac{K_j-K_1}{K_0-K_1}$. We define recursively:
\[
g_n = \frac{1}{K_0-K_n}(\partial(g_{n-1}) - K_n g_{n-1}) = h_{1,0} + \sum_{j \geq n+1} \gamma_{1,j}\sm{n} h_{1,j} \in \mathcal{I}.\mathcal{O}_p
\]
for constants $\gamma_{1,j}\sm{n}$. It is clear that $(g_n)_{n\in\mathbb{N}} \subset \mathcal{I}.\mathcal{O}_p$ converges formally to $h_{1,0}(x)$ (in the Krull topology of $\widehat{\mathcal{O}}_p$). By faithful flatness, this implies that $h_{1,0} \in \mathcal{I}.\mathcal{O}_p$ (see section 6.3 and Theorems 6.3.4 and 6.3.5 of $\cite{Hor}$ for a precise formulation of this result). It is now clear that $\mathcal{I}.\mathcal{O}_p = (h_{i,j})$ and we just need to use Noetherianity to obtain a finite system of generators.\\
\\
If $\partial$ is a regular vector-field then the result easily follows from an adaptation of the above argument. It is worth remarking that, in the later case, $K_i = 0$ for all $i$. 
\end{proof}
\subsection{$\theta$-Admissible Blowings-up}
\label{sec:Tadm}
Let $(M,\theta,E)$ be a $1$-foliated manifold and let $\mathcal{C}$ be an analytic sub-manifold of $M$. Consider the reduced ideal sheaf $\mathcal{I}_{\mathcal{C}}$ that generates $\mathcal{C}$, i.e.,$V(\mathcal{I}_{\mathcal{C}}) = \mathcal{C}$. We say that $\mathcal{C}$ is a $\theta$\textit{-admissible} center if:
\begin{itemize}
\item[$\bullet$] $\mathcal{C}$ is a regular closed sub-variety that has SNC with $E$;
\item[$\bullet$] Either $\mathcal{C}$ is $\theta$-invariant (i.e.,$\theta [\mathcal{I}_{\mathcal{C}}] \subset \mathcal{I}_{\mathcal{C}}$) or it is $\theta$-transverse (i.e.,$\theta[\mathcal{I}_{\mathcal{C}}] = \mathcal{O}_M$).
\end{itemize}
The objective of this definition is to avoid centers with finite tangency to the foliation. For example
\[
 \text{the center }\mathcal{I}_{\mathcal{C}} = (x^2-y,z) \text{ in } (\mathbb{R}^3,\partial_x,\emptyset) \text{ is \textbf{not} $\theta$-admissible.}
\]
In this context, it is clear that a blowing-up $\sigma: (\widetilde{M},\widetilde{\theta},\widetilde{E}) \longrightarrow (M,\theta,E)$ is $\theta$\textit{-admissible} if the center $\mathcal{C}$ is $\theta$-admissible. A \textit{sequence $\vec{\sigma}=(\sigma_1,...,\sigma_r)$ of $\theta$-admissible blowings-up} is a sequence of blowings-up:
\[
\begin{tikzpicture}
  \matrix (m) [matrix of math nodes,row sep=3em,column sep=3em,minimum width=2em]
  {(M_r,\theta_r,E_r) & \cdots & (M_1,\theta_1,E_1) & (M_0,\theta_0,E_0)\\};
  \path[-stealth]
    (m-1-1) edge node [above] {$\sigma_r$} (m-1-2)
    (m-1-2) edge node [above] {$\sigma_2$} (m-1-3)
    (m-1-3) edge node [above] {$\sigma_1$} (m-1-4);
\end{tikzpicture}
\]
such that $\sigma_i : (M_{i+1},\theta_{i+1},E_{i+1}) \to (M_{i},\theta_i,E_i)$ is a $\theta_i$-admissible blowing-up. Analogously, we can define sequence of $\theta$-invariant and $\theta$-transverse blowings-up.\\
\\
The following two Lemmas enlighten the interest of $\theta$-admissible blowings-up for the class of singular distribution which we are interested:
\begin{lemma}
 Let $(M,\theta,E)$ be a $1$-foliated manifold which is Log-Canonical (resp. $R$-monomial) and  $
\sigma: (\widetilde{M},\widetilde{\theta},\widetilde{E}) \longrightarrow (M,\theta,E)
$ be a $\theta$-transverse blowing-up. Then $\widetilde{\theta}$ is Log-Canonical (resp. $R$-monomial) and it is equal to $\mathcal{O}(-F) \sigma^{\ast}\theta$, where $F$ is the exceptional divisor created by $\sigma$.
\label{lem:thetatr}
\end{lemma}
\begin{proof}
 Since $\sigma$ is a $\theta$-transverse blowing-up, we know that $\theta[\mathcal{I}_{\mathcal{C}}] = \mathcal{O}_M$. So, at each point $p$ of $\mathcal{C}$ the singular distribution $\theta.\mathcal{O}_p$ is generated by a regular vector-field $\partial$. By the Flow-box Theorem, there exists a coordinate system $(x,\mathbf{y}) = (x,y_1,\dots,y_{m-1})$ such that $\partial = \partial_x$. Furthermore, there exists a function $f \in \mathcal{I}_{\mathcal{C}}.\mathcal{O}_p$ such that $\partial (f)$ is a unit. Taking $\bar{x}=f$ and $\bar{y}=y$, we obtain a coordinate system $(\bar{x},\bar{\mathbf{y}}) = (\bar{x},\bar{y}_1,\dots,\bar{y}_{m-1})$ where $\partial = \partial_{\bar{x}}$ and $\mathcal{I}_{\mathcal{C}}.\mathcal{O}_p$ is generated by $(\bar{x},\bar{y}_1,\dots,\bar{y}_t)$. Now, from the fact that $\mathcal{C}$ has SNC with $E$ and $\{\bar{x}=0\}$ must be transverse to $E$, apart from another change of coordinates in the $\bar{\mathbf{y}}$ coordinates, we can assume that $E$ is locally equal to $ \{ \Pi \bar{y}_i^{\epsilon_i} =0\} $, where $\epsilon_i\in \{0,1\}$. It is now clear that, after blowing-up, the singular distribution $\widetilde{\theta}$ is Log-Canonical (resp. $R$-monomial).
\end{proof}
\begin{lemma}
 Let $(M,\theta,E)$ be a $1$-foliated manifold which is Log-Canonical (resp. $R$-monomial) and  $
\sigma: (\widetilde{M},\widetilde{\theta},\widetilde{E}) \longrightarrow (M,\theta,E)
$ be a $\theta$-invariant blowing-up. Then $\widetilde{\theta}$ is Log-Canonical (resp. $R$-monomial) and it is equal to $\sigma^{\ast}\theta$.
\label{lem:thetainv}
\end{lemma}
\begin{proof}
 Since $\sigma$ is a $\theta$-invariant blowing-up, we know that $\theta[\mathcal{I}_{\mathcal{C}}] \subset \mathcal{I}_{\mathcal{C}}$. As a first remark, let $\theta^{\ast} = \sigma^{\ast}\theta$ and $F$ be the exceptional divisor introduced by $\sigma$. Then:
\[
 \theta^{\ast} [ \mathcal{O}(-F) ] = \theta^{\ast} [ \mathcal{I}^{\ast} ] \subset \sigma^{\ast}\{\theta [ \mathcal{I} ] \} + \mathcal{I}^{\ast} = \mathcal{I}^{\ast} = \mathcal{O}(-F) 
\]
which implies that $\theta^{\ast}$ is tangent to $F$ and, by consequence, to $\widetilde{E}$. Now, let $p$ be a point in the center $\mathcal{C}$. We study the preimage of $p$ and, to that end, we divide in two cases:
 \medskip

\noindent{\emph{Case 1 - $\theta.\mathcal{O}_p$ is regular:}} In other words, the vector-field $\partial$ that generated $\theta.\mathcal{O}_p$ is a regular vector-field. Let us fix a monomial coordinate system $(x,\mathbf{y})= (x,y_1,\dots,y_{m-1})$ centered at $p$ such that $\partial_x$ generated $\theta . \mathcal{O}_p$. Then, by Lemma \ref{lem:InvNSing}, there exists a set of generators $\{f_1(\mathbf{y}),\dots, f_n(\mathbf{y})\}$ of $\mathcal{I}_{\mathcal{C}}.\mathcal{O}_p$ which is independent of the $x$-coordinate. Finally, since $\mathcal{I}_{\mathcal{C}}$ has SNC with $E$, apart from a change of coordinates in the $\mathbf{y}$ coordinates, we can assume that:
\[
\begin{aligned}
 \theta\phantom{_{\mathcal{C}}} & \text{ is locally generated by } \partial_x \text{, }\\
 \mathcal{I}_{\mathcal{C}} &\text{ is locally generated by } (y_1,\dots,y_t) \text{ for some }t\leq m-1 \text{ and }\\
 E\phantom{_{\mathcal{C}}} & \text{ is locally generated equal to } \{ \Pi y_i^{\epsilon_i} =0\} \text{ where }\epsilon_i\in \{0,1\}
\end{aligned}
\]
 In this case, at every point $q$ in the preimage of $p$ by $\sigma$, it is clear that $\theta^{\ast}.\mathcal{O}_q$ is a regular singular distribution. Since $\theta^{\ast}$ is tangent to $\widetilde{E}$, we conclude that $\widetilde{\theta}.\mathcal{O}_q=\theta^{\ast}.\mathcal{O}_q$ and, thus, it is Log Canonical (resp. $R$-monomial).  
\bigskip

\noindent{\emph{Case 2 - $\theta.\mathcal{O}_p$ is singular:}} In this case, it is convenient to divide in two cases, depending on the singularity class we are considering:
\bigskip

\noindent{\emph{Case 2.1 - $\theta$ is Log Canonical:}} Let us fix a generator $\partial$ of $\theta.\mathcal{O}_p$ and a coordinate system $\mathbf{x} = (x_1,\dots,x_m)$ such that
\[
\begin{aligned}
 \mathcal{I}_{\mathcal{C}} & \text{ is locally generated by } (x_1,\dots,x_t) \text{ for some }t\leq m \text{ and }
\end{aligned}
\]
we have that $\partial = \sum_i A_i \partial_{x_i}$, where $A_i \in \boldsymbol{m}_p$ and $\bar{\partial}$ is non-nilpotent, where we recall that the function $\bar{\partial}$ is given by:
\[
  \bar{\partial} :  \frac{\boldsymbol{m}_p}{\boldsymbol{m}^2_p} \to  \frac{\boldsymbol{m}_p}{\boldsymbol{m}^2_p}
\]
In other words, if we denote the linear part of $A_i$ by $\sum_j a_{i,j} x_i$, then the matrix:
\[
 A = \begin{bmatrix}
      a_{1,1} & \dots & a_{1,m}\\
      \vdots& \ddots & \vdots \\
      a_{m,1}& \dots & a_{m,m}
     \end{bmatrix} = 
     \begin{bmatrix} a_{1,1} & v_1 & v_2\\
      w_1 & B_{1,1} & B_{1,2} \\
      w_2& B_{2,1} & B_{2,2}
      \end{bmatrix}
\]
is non-nilpotent, where $B_{1,1}$ is a $(t-1)\times (t-1)$ matrix. Since the function $A_j$ with $j\leq t$ belongs to $\mathcal{I}_{\mathcal{C}}.\mathcal{O}_p = (x_1,\dots, x_t)$, we conclude that $v_2=0$ and $B_{1,2}=0$.\\
\\
Now, let $q$ be a point in the preimage of $p$ by $\sigma$. Without loss of generality, we can assume that $q$ is the origin of the $x_1$ chart and that:
\[
 \sigma^{\ast}\partial_p = A_1^{\ast}\partial_{\widetilde{x}_1} + \frac{1}{\widetilde{x}_1}\sum^t_{j=2} A_j^{\ast} - A_1^{\ast}\widetilde{x}_i \partial_{\widetilde{x}_i} + \sum_{j>t} A_j \partial_{\widetilde{x}_i}
\]
since by hypothesis the function $A_j$ with $j\leq t$ belongs to $\mathcal{I}_{\mathcal{C}} = (x_1,\dots, x_t)$, we conclude that $\frac{A_j^{\ast}}{\widetilde{x}_1}$ is holomorphic for all $i\leq t$, which implies that $\sigma^{\ast}\partial$ is holomorphic. Furthermore:
\begin{itemize}
\item If $w_1$ is a non-zero $(t-1)\times 1$ matrix, then $\sigma^{\ast}\partial$ is a regular vector-field (because there exists $j$ such that $\frac{A_j^{\ast}}{\widetilde{x}_1}$ is a unit). In this case, it is clear that $\theta^{\ast}.\mathcal{O}_q = \widetilde{\theta}.\mathcal{O}_q$ and that $\widetilde{\theta}$ is Log canonical at $q$;
\item If $w_1 =0 $, then $\sigma^{\ast}\partial$ is singular. In this case, let us notice that:
\[
 A =
     \begin{bmatrix} a_{1,1} & v_1 &0\\
      0 & B_{1,1} & 0\\
      w_2& B_{2,1} & B_{2,2}
      \end{bmatrix}
\]
is a non-nilpotent matrix, which implies that either $[a_{1,1}]$ or $B_{1,1}$ or $B_{2,2}$ is not a nilpotent matrix. Now, the linear part of $\sigma^{\ast}$ is given by the following matrix:
\[
 \widetilde{A}
=     \begin{bmatrix} a_{1,1} & 0 & 0\\
      * & B_{1,1} - a_{1,1}Id & * \\
      w_2& 0 & B_{2,2}
      \end{bmatrix}
\]
where $Id$ is the $(t-1)\times (t-1)$ identity matrix and $*$ stands for a part of the matrix which is not dependent of the matrix $A$ (because it depends on the quadratic terms of $A_j$ for $1<j<t$). In one hand, if $a_{1,1}\neq 0$, then it is clear that $\widetilde{A}$ is not nilpotent and that we can not factor $\widetilde{x}_1$ out of the vector-field $\sigma^{\ast}\partial$ (otherwise the vector-field would not be tangent to $F = \{\widetilde{x}_1=0\}$). At another hand, if $a_{1,1}=0$, then either $B_{1,1}$ or $B_{2,2}$ is non-nilpotent which implies that $\widetilde{A}$ is non-nilpotent. Furthermore, we clearly can not factor $\widetilde{x}_1$ out of $\sigma^{\ast}\partial$. So, in particular,  $\theta^{\ast}.\mathcal{O}_q = \widetilde{\theta}.\mathcal{O}_q$ and $\widetilde{\theta}$ is Log canonical at $q$.
\end{itemize}
\noindent{\emph{Case 2.2 - $\theta$ is $R$-monomial:}} Let us fix a generator $\partial$ of $\theta.\mathcal{O}_p$ and a monomial coordinate system $\mathbf{x} = (x_1,\dots,x_m)$. Notice that, if $\mathcal{C}\subset E^{(i)}$, where $E^{(i)}$ is a irreducible component of the exceptional divisor $E$ locally given by $\{x_i=0\}$, then the function $x_i$ belongs to $\mathcal{I}_{\mathcal{C}}.\mathcal{O}_p$.\\
\\
So, there exists a set of generators $(x_{i_1},\dots,x_{i_r},f_1, \dots f_n)$ that generates $\mathcal{I}_{\mathcal{C}}.\mathcal{O}_p$, where $\mathcal{C} \subset E^{(i_j)}$ for all $j\leq r$ and $f_i$ is a functions independent on the $x_{i_j}$ coordinates. Furthermore, by Lemma $\ref{lem:InvNSing}$, we can assume that $\partial f_i = K_i f_i$, where $K_i$ is a constant in the ring $R$.\\
\\
Now, since $\mathcal{I}_{\mathcal{C}}$ is regular, we can suppose that $f_1$ is regular and, without loss of generality, that  $\frac{\partial}{\partial x_1}f(p) \neq 0$. Furthermore, since $\mathcal{C}$ has SNC with $E$, we can assume that $\{x_1 =0\}$ is not a divisor of $E$. So, let us consider the change of coordinates $\bar{x}_1=f_1$ and $\bar{x}_i=x_i$ otherwise. In the new coordinates, we get:
\[
\partial= K_{1} \bar{x}_1 \frac{\partial}{\partial \bar{x}_1} + \sum^m_{i=2} \alpha_{i} \bar{x}_j \frac{\partial}{\partial \bar{x}_i}
\]
because $\partial(f_1) = K_1 f_1$. So, this coordinate is still monomial, since the divisor $E$ is still locally given by $\{\Pi \bar{x}_i^{\epsilon_i}=0\}$. Repeating this process a finite number of times, we get a coordinate $\mathbf{y} = (y_1,\dots,y_m)$ such that
\[
\begin{aligned}
 \theta\phantom{_{\mathcal{C}}} & \text{ is locally generated by } \sum \gamma_i y_i \partial_{y_i} \text{ where }\gamma_i \in R\text{, }\\
 \mathcal{I}_{\mathcal{C}} &= (y_1,\dots,y_t) \text{ for some }t\leq m \text{ and }\\
 E\phantom{_{\mathcal{C}}} & \text{ is locally generated equal to } \{ \Pi y_i^{\epsilon_i} =0\} \text{ where }\epsilon_i\in \{0,1\}
\end{aligned}
\]
It is now clear that, after blowing-up, the singular distribution $\widetilde{\theta}$ is also $R$-monomial. 
\end{proof}
These two Lemmas, clearly imply the following Proposition:
\begin{proposition}
Let $(M,\theta,E)$ be a $1$-foliated manifold which is Log-Canonical (resp. $R$-monomial) and  $
\sigma: (\widetilde{M},\widetilde{\theta},\widetilde{E}) \longrightarrow (M,\theta,E)
$ be a $\theta$-admissible blowing-up. Then $\widetilde{\theta}$ is Log-Canonical (resp. $R$-monomial).
\label{prop:AdmCenter}
\end{proposition}
\section{Foliated Ideal Sheaf}
\subsection{The tangency chain}
\label{sec:CIS}
A chain of ideal sheaves consists of a sequence $(\mathcal{I}_i)_{i \in \mathbb{N}}$ of ideal sheaves such that $\mathcal{I}_i \subset \mathcal{I}_j$ if $i \leq j$. The \textit{length} of a chain of ideal sheaves at a point $p$ of $M$ is the minimal number $\nu_p \in \mathbb{N}$ such that $\mathcal{I}_i.\mathcal{O}_{p}=\mathcal{I}_{\nu_p}.\mathcal{O}_{p}$ for all $i \geq \nu_p$. We distinguish two cases:
\[
 \begin{array}{l}
  \text{If $\mathcal{I}_{\nu_p}.\mathcal{O}_{p}= \mathcal{O}_{p}$, then the chain is said to be of type $1$ at $p$;}\\
  \text{If $\mathcal{I}_{\nu_p}.\mathcal{O}_{p} \neq \mathcal{O}_{p}$, then the chain is said to be of type $2$ at $p$.}
 \end{array}
\]
Moreover, fixed a chain of ideal sheaf $(\mathcal{I}_n)$, it is not difficult to see that the functions $\nu: M \to \mathbb{N}$ and $type: M \to \{1,2\}$ are upper semi-continuous (where $\nu(p) = $ the length of $(\mathcal{I}_n)$ at $p$ and $type(p) = $ the type of $(\mathcal{I}_n)$ at $p$). So, given a subset $U$ of $M$, the definition of length and type naturally extends to $U$ as follows:
\[
 \begin{array}{l}
  \text{The length of $(\mathcal{I}_n)$ at $U$ is $\nu_U := sup\{ \nu_p ;  p \in U\}$;}\\
  \text{The type of $(\mathcal{I}_n)$ at $U$ is $type_U := sup\{ type_p ;  p \in U\}$.}
 \end{array}
\]
Notice that $\nu_U$ may be infinity but, if $U$ is a relatively compact open subset of $M$, $\nu_U$ is necessarily finite.\\
\\
Now, given a foliated ideal sheaf $(M,\theta,\mathcal{I},E)$, the \textit{tangency chain} of $(M,\theta,\mathcal{I},E)$ is a chain of ideals $\mathcal{T}g(\theta,\mathcal{I})= \{H(\theta,\mathcal{I},i)\}_{ i \in \mathbb{N}}$ defined by:
\[ \left\{
\begin{aligned}
H(\theta,\mathcal{I},0) \phantom{+1} &:= \mathcal{I}  \\
H(\theta,\mathcal{I},i+1) &:= H(\theta, \mathcal{I},i)+ \theta[H(\theta,\mathcal{I},i)]  
\end{aligned}
\right.
\]
At each point $p$ in $M$, the length of this chain is called the \text{tangent order} (or shortly, the \textit{tg-order}) at $p$, and is denoted by $\nu_p(\theta,\mathcal{I})$. The type of the chain is denoted by $type_p(\theta,\mathcal{I})$.
\begin{remark}
Suppose that $\theta$ is generated by a regular vector-field $\partial$ and let $\gamma_p$ be the orbit of $\partial$ passing through a point $p$ of $V(\mathcal{I})$. In this simple case, we can interpret these invariants as follow:
\begin{itemize}
\item[$\bullet$] If the orbit $\gamma_p$ is contained in the variety $V(\mathcal{I})$, then the type of $(\theta,\mathcal{I})$ at $p$ is two;
\item[$\bullet$] If the orbit $\gamma_p$ is \textbf{not} contained in $V(\mathcal{I})$, then the type of $(\theta,\mathcal{I})$ at $p$ is one. Furthermore, the $tg$-order of $(\theta,\mathcal{I})$ is equal to the order of tangency between the orbit $\gamma_p$ and the variety $V(\mathcal{I})$ at $p$.
\end{itemize}
In other words, the type identifies the presence of invariant leaves and the $tg$-order measures the order of tangency between the leaves and the variety $V(\mathcal{I})$.
\end{remark}
\begin{lemma}
Let $\sigma:(\widetilde{M},\widetilde{\theta},\widetilde{E}) \rightarrow (M,\theta,E)$ be a $\theta$-admissible blowing-up over a $1$-foliated Log-Canonical (resp. $R$-monomial) ideal sheaf $(M,\theta,\mathcal{I},E)$, where $F$ is the divisor created by $\sigma$. Then:
\begin{itemize}
\item[$i)$] If the blowing-up is $\theta$-invariant, then $[\theta(\mathcal{I})]^{\ast} +  \mathcal{I}^{\ast} = \widetilde{\theta}(\mathcal{I}^{\ast}) + \mathcal{I}^{\ast}$.
\item[$ii)$] If the blowing-up is $\theta$-transverse, then $[\theta(\mathcal{I})]^{\ast} \mathcal{O}(-F) +  \mathcal{I}^{\ast} = \widetilde{\theta}(\mathcal{I}^{\ast}) + \mathcal{I}^{\ast}$;
\end{itemize}
\label{lem:AlgProp}
\end{lemma}
\begin{proof}
Notice that, since $\sigma^{\ast}: Der_{M} \rightarrow BlDer_{\widetilde{M}}$ is a morphism, it is clear that $[\theta(\mathcal{I})]^{\ast} \subset \theta^{\ast}(\mathcal{I}^{\ast})$, which implies that:
\begin{itemize}
\item If the blowing-up is $\theta$-invariant, by Lemma \ref{lem:thetainv} $\widetilde{\theta} = \theta^{\ast}$ and, thus:
\[
[\theta(\mathcal{I})]^{\ast} +  \mathcal{I}^{\ast} \subset \theta^{\ast}(\mathcal{I}^{\ast}) + \mathcal{I}^{\ast} = \widetilde{\theta}(\mathcal{I}^{\ast}) + \mathcal{I}^{\ast}
\]
\item If the blowing-up is $\theta$-transverse, by Lemma \ref{lem:thetatr} $\widetilde{\theta} = \mathcal{O}(-F) \theta^{\ast}$ and, thus:
\[
 [\theta(\mathcal{I})]^{\ast} \mathcal{O}(-F) +  \mathcal{I}^{\ast} \subset 
[\theta^{\ast}(\mathcal{I}^{\ast})] \mathcal{O}(-F) +  \mathcal{I}^{\ast} =
\widetilde{\theta}(\mathcal{I}^{\ast}) + \mathcal{I}^{\ast}
\]
\end{itemize}
To prove the other side of the inclusion, fix a point $q$ in $\widetilde{M}$ and a function $G$ in $\widetilde{\theta}(\mathcal{I}^{\ast}).\mathcal{O}_q$. We can write $G$ as:
\[
G = A [x^{\epsilon_i} \partial^{\ast}] \left( \sum_i b_{i} g_{i}^{\ast} \right)
\]
where $x$ is a generator of $\mathcal{O}(-F)$; $\partial$ is a derivation that generated $\theta . \mathcal{O}_p$; the coeficient $\epsilon$ is $0$ or $1$; $x^{\epsilon} \partial^{\ast}$ is an holomorphic derivation; the functions $g_{i}$ belongs to $\mathcal{I}$ and $A$ and $b_{i}$ are general germs over $\mathcal{O}_q$. In particular
\[
 G =  A \sum_{i} x^{\epsilon} \partial^{\ast}[b_{i}] g_{i}^{\ast} + M^{\epsilon} \partial^{\ast}[g_{i}^{\ast}] b_{i} 
\]
Since $x^{\epsilon} \partial^{\ast}[b_{i}] g_{i}^{\ast} \in \mathcal{I}^{\ast}$ and $x^{\epsilon} \partial^{\ast}[g_{i}^{\ast}] b_{i} \in \mathcal{O}(-\epsilon F) [\theta(\mathcal{I})]^{\ast}$ we conclude that:
\[
 G \in  \mathcal{I}^{\ast} + \mathcal{O}(-\epsilon F) [\theta(\mathcal{I})]^{\ast}
\]
and, thus $\widetilde{\theta}(\mathcal{I}^{\ast}) \subset \mathcal{I}^{\ast} + \mathcal{O}(-\epsilon F) [\theta(\mathcal{I})]^{\ast}$ We now just need to remark that:
\begin{itemize}
 \item  If the blowing-up is $\theta$-invariant, then $\epsilon=0$ and
\[
 \widetilde{\theta}(\mathcal{I}^{\ast}) + \mathcal{I}^{\ast} \subset [\theta(\mathcal{I})]^{\ast} +  \mathcal{I}^{\ast} 
\]
 \item  If the blowing-up is $\theta$-transverse then $\epsilon =1$ and
 \[
 \widetilde{\theta}(\mathcal{I}^{\ast}) + \mathcal{I}^{\ast} \subset
[\theta(\mathcal{I})]^{\ast} \mathcal{O}(-F) +  \mathcal{I}^{\ast}
\]
\end{itemize}
\end{proof}
\subsection{Chain-preserving smooth morphism}
\label{sec:SMCPM}
Given two $1$-foliated ideal sheaves $(M,\theta,\mathcal{I},E_M)$ and $(N,\omega,\mathcal{J},E_N)$, a smooth morphism $\phi: M \to N$ is said to be \textit{smooth in respect to $(M,\theta,\mathcal{I},E_M)$ and $(N,\omega,\mathcal{J},E_N)$} if $\mathcal{J}.\mathcal{O}_M=\mathcal{I}$ and $\phi^{-1}(E_{N}) = E_{M}$ (see more details in \cite{Ko}). In this case, we abuse notation and denote the morphism as:
\[
\phi: (M,\theta,\mathcal{I},E_M) \to (N,\omega,\mathcal{J},E_N)
\]
Notice that this definition is independent of the singular distributions $\theta$ and $\omega$ (besides the fact that both $\theta$ and $\omega$ have the same leaf dimension). We say that a smooth morphism $\phi: (M,\theta,\mathcal{I},E_M) \longrightarrow (N,\omega,\mathcal{J},E_N)$ is \textit{chain-preserving} if:
\[
\mathcal{T}g(\omega,\mathcal{J}) . \mathcal{O}_M = \mathcal{T}g(\theta,\mathcal{I})
\]
i.e.,$H(\omega,\mathcal{J},i),\mathcal{O}_M = H(\theta,\mathcal{I},i)$ for all $i \in \mathbb{N}$. We also extend the definition to local foliated ideal sheaves in the natural way i.e.,a morphism $\phi:(M,M_0,\theta,\mathcal{I},E_M) \to (N,N_0,\omega,\mathcal{J},E_N)$ is chain-preserving if: 
\[
\phi|_{M_0} : (M_0,\theta.\mathcal{O}_{M_0},\mathcal{I}.\mathcal{O}_{M_0},E_M \cap M_0) \to (N_0,\omega.\mathcal{O}_{N_0},\mathcal{J}.\mathcal{O}_{N_0},E_N \cap N_0)
\]
is chain-preserving.
\subsection{Blow-up Functors}
\label{sec:RLU}
We follow the presentation of \cite{Ko} (more specifically, definition 3.31 of \cite{Ko}) in order to define the notion of a \textit{blow-up functor} $\mathcal{S}$ that has:
\begin{itemize}
\item[$\bullet$] input: The category whose objects are local foliated ideal sheaves $(M,M_0,\theta, \allowbreak \mathcal{I},E)$ and whose morphisms are smooth morphisms;
\item[$\bullet$] output: The category whose objects are admissible blowing-up sequences:
\[
\begin{tikzpicture}
  \matrix (m) [matrix of math nodes,row sep=3em,column sep=3em,minimum width=2em]
  {
     (M_r,\theta_r,\mathcal{I}_r,E_r) & \cdots & (M_1,\theta_1,\mathcal{I}_1,E_1) & (M_0,\theta_0,\mathcal{I}_0,E_0)\\};
  \path[-stealth]
    (m-1-1) edge node [above] {$\sigma_r$} (m-1-2)
    (m-1-2) edge node [above] {$\sigma_2$} (m-1-3)
    (m-1-3) edge node [above] {$\sigma_1$} (m-1-4);
\end{tikzpicture}
\]
with specified admissible centers $\mathcal{C}_i$ and whose morphisms are given by the Cartesian product.
\end{itemize}
\begin{remark}
For such a functor to be well-defined, we accept blowings-up with empty centers (isomorphisms).
\end{remark}
A blow-up functor $\mathcal{S}$ is said to be a \textit{chain-preserving blow-up Functor}, if the morphisms in the input and output are all chain-preserving morphisms.
\subsection{Resolution of singularities}
Let us state the version of Hironaka`s Theorem that we are going to use. We remark that we do not use the standard notation, since we use the notion of foliated ideal sheaf:
\begin{theorem}
(Hironaka) Let $(M,M_0,\theta,\mathcal{I},E)$ be a local foliated ideal sheaf. There exists a sequence of admissible blowings-up of order one of $(M_0,\theta,\mathcal{I},E)$:
\[
\begin{tikzpicture}
  \matrix (m) [matrix of math nodes,row sep=3em,column sep=3em,minimum width=1em]
  {\mathcal{R}(M,M_0,\theta,\mathcal{I},E): (M_r,\theta_r,\mathcal{I}_r,E_r) & \cdots & (M_0,\theta_0,\mathcal{I}_0,E_0)\\};
  \path[-stealth]
    (m-1-1) edge node [above] {$\sigma_r$} (m-1-2)
    (m-1-2) edge node [above] {$\sigma_1$} (m-1-3);
\end{tikzpicture}
\]
such that $\mathcal{I}_r = \mathcal{O}_{M_r}$. Furthermore, $\mathcal{R}$ is a blowing-up functor.
\label{th:Hironaka}
\end{theorem}
\begin{remark}
 Notice that there is no claim upon the singular distribution $\theta_r$.
\end{remark}
\begin{remark}
The above Theorem is an interpretation of Theorem 1.3 of $\cite{BM2}$ or Theorems 2.0.3 and 6.0.6 of $\cite{Wo}$ in the following sense:
\begin{itemize}
\item[$\bullet$] Theorem $1.3$ of $\cite{BM2}$ is enunciated in algebraic category. But the paragraph before Theorem $1.1$ of $\cite{BM2}$ justifies the analytic statement;
\item[$\bullet$] In $\cite{BM2}$ and $\cite{Wo}$, the authors work with \textit{marked ideal sheaves}. We specialize their result to marked ideal sheaves with weight one;
\item[$\bullet$] In order to stress the functorial property of the resolution, we have followed Kollor's presentation (see \cite{Ko}, definition 3.31). 
\end{itemize}
\end{remark}
%
%
%
%
\section{Resolution of an invariant ideal sheaf}
\subsection{Statement of the result}
In this section we prove the following result:
\begin{proposition}
\label{th:HironakaS}
Let $(M,M_0,\theta, \mathcal{I},E)$ be a local Log-Canonical (resp. $R$-monomial) $1$-foliated ideal sheaf. Suppose that $\mathcal{I}_0$ is invariant by $\theta_0$, i.e.,$\theta [ \mathcal{I}] . \mathcal{O}_{M_0} \allowbreak \subset \mathcal{I} . \mathcal{O}_{M_0}$. Then, there exists a sequence of $\theta$-invariant blowings-up of order one:
\[
\begin{tikzpicture}
  \matrix (m) [matrix of math nodes,row sep=3em,column sep=3em,minimum width=1em]
  {\mathcal{R}_{inv}(M,M_0,\theta,\mathcal{I},E): (M_r,\theta_r,\mathcal{I}_r,E_r) & \cdots & (M_0,\theta_0,\mathcal{I}_0,E_0)\\};
  \path[-stealth]
    (m-1-1) edge node [above] {$\sigma_r$} (m-1-2)
    (m-1-2) edge node [above] {$\sigma_1$} (m-1-3);
\end{tikzpicture}
\]
such that:
\begin{itemize}
\item[i]) The ideal sheaf $\mathcal{I}_r$ is equal to the structural ring $\mathcal{O}_{M_r}$. In particular, the pull-back of $\mathcal{I}$ is principal ideal sheaf with support contained in $E_r$;
\item[ii]) The singular distribution $\theta_r$ is Log-Canonical (resp. $R$-monomial);
\item[iii]) $\mathcal{R}_{inv}$ is a chain-preserving blow-up functor (i.e.,it commutes with chain-preserving smooth morphisms).
\end{itemize}
\end{proposition}
This result is a consequence of the Functoriality of Hironaka`s Theorem (see, e.g \cite{BM2}). In what follows we present a rigorous proof.
%
%
%
%
%
%
%
%
%
\subsection{Proof of Proposition \ref{th:HironakaS}}
In order to prove the Proposition, we introduce the notion of geometrical invariance: an ideal sheaf $\mathcal{I}$ is \textit{geometrically invariant} by $\theta$ if every leaf of the foliation generated by $\theta$ that intersects $V(\mathcal{I})$ is totally contained in $V(\mathcal{I})$.
\begin{lemma}
Let $\theta$ be an involutive $1$-singular distribution and $\mathcal{I}$ a reduced ideal sheaf. Then $\mathcal{I}$ is $\theta$-invariant if, and only if, $\mathcal{I}$ is geometrically invariant by $\theta$.
\label{lem:Fol}
\end{lemma}
\begin{proof}
To prove the if implication, let $p$ be a point in $V(\mathcal{I})$, $\mathcal{L}$ be the leaf of $\theta$ passing through $p$ and $\partial$ be a vector-field that generates $\theta.\mathcal{O}_p$. If $\mathcal{L}$ is zero-dimension, it is clear that $\mathcal{L} \subset V(\mathcal{I})$, so we may assume that $\mathcal{L}$ is one-dimensional. In this case, there exists a coordinate system $(x,\mathbf{y}) = (x,y_1,\dots,y_{m-1})$ such that $\partial = \partial_x$. By Lemma \ref{lem:InvNSing}, there exists a system of generators of $\mathcal{I}$ which is independent of the $x$ coordinate and the result is now obvious.\\
\\
To prove the only if implication, let us assume that $\mathcal{I}$ is a reduced ideal sheaf which is geometrically invariant by $\theta$. We claim that $V(\mathcal{I}) \subset V(\theta(\mathcal{I}))$, which is enough to prove the Lemma since it would imply that:
\[
\theta(\mathcal{I}) \subset \sqrt{\theta(\mathcal{I})} \subset \sqrt{\mathcal{I}} = \mathcal{I}
\]
In order to prove the claim, let $p$ be a point of $ V(\mathcal{I})$, $\mathcal{L}$ be the leaf of $\theta$ passing through $p$, $f$ be an arbitrary function in $\mathcal{I}$ and $\partial$ be a vector-field that generates $\theta.\mathcal{O}_p$. Since, by hypothesis, $\mathcal{L} \subset V(\mathcal{I})$, we conclude that $f|_{\mathcal{L}} \equiv 0$. Moreover, since $\partial$ is tangent to $\mathcal{L}$, we conclude that $\partial(f)|_{\mathcal{L}} = \partial|_{\mathcal{L}}(f|_{\mathcal{L}}) =0$ and, in particular, $p \in V(\partial(f))$. Since the choice of $f \in \mathcal{I}$ was arbitrary, we conclude that $p$ belongs to $\theta(\mathcal{I})$ as we wanted to prove.
\end{proof}
We now state a preliminary Lemma:
\begin{lemma}
Let $(M,\theta,\mathcal{I},E)$ be a Log-Canonical (resp. $R$-monomial) foliated ideal sheaf and let us consider a $\theta$-invariant blowing-up of order one $\sigma: (\widetilde{M},\widetilde{\theta},\widetilde{\mathcal{I}},\widetilde{E}) \longrightarrow (M,\theta,\mathcal{I},E)$. Then $\widetilde{\mathcal{I}}$ is invariant by $\widetilde{\theta}$.
\label{lem:Hir1}
\end{lemma}
\begin{proof}
Since $\mathcal{C}$ is a regular sub-manifold geometrically invariant by $\theta$, by Lemma $\ref{lem:Fol}$ the ideal sheaf $\mathcal{I}_{\mathcal{C}}$ is $\theta$-invariant. Furthermore, by Lemma \ref{lem:AlgProp} we have that:
\[
\begin{array}{ccc}
\theta [ \mathcal{I}_{\mathcal{C}}] \subset \mathcal{I}_{\mathcal{C}} & \Rightarrow & \theta^{\ast}[\mathcal{O}(-F)] \subset \mathcal{O}(-F)
\end{array}
\]
Now, by Lemma $\ref{lem:thetainv}$, $\widetilde{\theta}=\theta^{\ast}$. Thus, again by Lemma \ref{lem:AlgProp}:
\[
\begin{aligned}
\widetilde{\theta}[\widetilde{\mathcal{I}}] + \widetilde{\mathcal{I}} &= \theta^{\ast}[\mathcal{I}^{\ast} . \mathcal{O}(F)] + \widetilde{\mathcal{I}} \\
&\subset \theta^{\ast}[ \mathcal{I}^{\ast}] \mathcal{O}(F) + \mathcal{I}^{\ast} \theta^{\ast}[\mathcal{O}(F)] + \widetilde{\mathcal{I}} = \widetilde{\mathcal{I}}
\end{aligned}
\]
which concludes the Lemma.
\end{proof}
Now, we are ready to start the proof of Proposition \ref{th:HironakaS}.
\begin{proof}
By the Hironaka's Theorem $\ref{th:Hironaka}$, there exists a resolution of singularities $\vec{\sigma}=(\sigma_1,...,\sigma_r)$ of $(M,M_0,\theta,\mathcal{I},E)$:
\[
\begin{tikzpicture}
  \matrix (m) [matrix of math nodes,row sep=3em,column sep=3em,minimum width=1em]
  {\mathcal{R}(M,M_0,\theta,\mathcal{I},E): (M_r,\theta_r,\mathcal{I}_r,E_r) & \cdots & (M_0,\theta_0,\mathcal{I}_0,E_0)\\};
  \path[-stealth]
    (m-1-1) edge node [above] {$\sigma_r$} (m-1-2)
    (m-1-2) edge node [above] {$\sigma_1$} (m-1-3);
\end{tikzpicture}
\]
where $\sigma_i : (M_i,\theta_i,\mathcal{I}_i,E_i) \longrightarrow (M_{i-1},\theta_{i-1},\mathcal{I}_{i-1},E_{i-1})$ has center $\mathcal{C}_i$.
\begin{claim}
\label{cl:inv1}
 The sequence of blowings-up $\vec{\sigma}=(\sigma_1,...,\sigma_r)$ is $\theta$-invariant.
\end{claim}
\begin{proof}
Suppose by induction that the centers $\mathcal{C}_i$ are $\theta_{i-1}$-invariant for $i<k$. We need to verify that $\mathcal{C}_k$ is also $\theta_{k-1}$-invariant (including for $k=1$). Since $\mathcal{C}_k$ is regular, by Lemma $\ref{lem:Fol}$, we only need to verify that $\mathcal{C}_k$ is geometrically invariant by $\theta_{k-1}$.\\
\\
To this end, let $\mathcal{L}$ be a connected leaf of $\theta_{k-1}$ with non-empty intersection with $\mathcal{C}_{k}$. If $\mathcal{L}$ is zero-dimensional, it is clear that  $\mathcal{L} \subset \mathcal{C}_{k}$, so let us assume that $\mathcal{L}$ is one-dimensional. Let $p$ be a point in the intersection $\mathcal{C}_k \cap \mathcal{L}$ and $\partial$ be a generator of $\theta_{k-1}.\mathcal{O}_p$. Since $\partial$ is a regular vector-field, there exists a coordinate system $(x,\mathbf{y}) = (x,y_1,...,y_{m-1})$ centered at $p$ such that $\partial = \partial_x$. Now, notice that $\mathcal{I}_{k-1}$ is $\theta_{k-1}$-invariant because of the induction hypotheses and a recursive use Lemma \ref{lem:Hir1}. So, by Lemma \ref{lem:InvNSing}, there exists a set of generators $\{f_1(\mathbf{y}),...,f_n(\mathbf{y})\}$ of $\mathcal{I}_{k-1}. \mathcal{O}_p$ which is independent of the coordinate $x$. It is now clear that the functorial statement of Hironaka's Theorem $\ref{th:Hironaka}$ guarantees that the center $\mathcal{C}_k$ is locally geometrically invariant by $\partial$. Since the choice of $p$ in the intersection $\mathcal{C}_k \cap \mathcal{L}$ was arbitrarily, we conclude that $\mathcal{L} \subset \mathcal{C}_k$, which ends the proof.
\end{proof}
Now, the functoriality statement of Theorem $\ref{th:HironakaS}$ is a direct consequence of the functoriality of Theorem $\ref{th:Hironaka}$ and the Log-Canonicity (resp. $R$-monomiality) statement is a direct consequence of Proposition $\ref{prop:AdmCenter}$. This concludes the proof of Proposition \ref{th:HironakaS}.
\end{proof}
\section{Resolution of singularities subordinated to a $1$-foliation}
\subsection{Statement of the result}
In this section we prove our main Theorem:
\begin{theorem}
Let $(M,M_0,\theta, \mathcal{I},E)$ be a local Log-Canonical (resp. $R$-monomial) $1$-foliated ideal sheaf. Then, there exists a sequence of $\theta$-admissible blowings-up of order one:
\[
\begin{tikzpicture}
  \matrix (m) [matrix of math nodes,row sep=3em,column sep=3em,minimum width=1em]
  {\mathcal{R}_1(M,M_0,\theta,\mathcal{I},E): (M_r,\theta_r,\mathcal{I}_r,E_r) & \cdots & (M_0,\theta_0,\mathcal{I}_0,E_0)\\};
  \path[-stealth]
    (m-1-1) edge node [above] {$\sigma_r$} (m-1-2)
    (m-1-2) edge node [above] {$\sigma_1$} (m-1-3);
\end{tikzpicture}
\]
such that:
\begin{itemize}
\item[i]) The ideal sheaf $\mathcal{I}_r$ is equal to the structural ring $\mathcal{O}_{M_r}$. In particular, the pull-back of $\mathcal{I}$ is principal ideal sheaf with support contained in $E_r$;
\item[ii]) The singular distribution $\theta_r$ is Log-Canonical (resp. $R$-monomial);
\item[iii]) $\mathcal{R}_1$ is a chain preserving blow-up functor (i.e.,it commutes with chain-preserving smooth morphisms).
\end{itemize}
\label{th:Hi1}
\end{theorem}
We remark that the sequence of blowings-up in Theorem \ref{th:Hi1} is \textbf{different} from the one given by Hironaka`s Theorem.
\subsection{Proof of Theorem \ref{th:Hi1}}
Let us start giving the intuitive idea of the proof. Given a local $1$-foliated ideal sheaf $(M,M_0,\theta,\mathcal{I},E)$ the main invariant we consider is the pair:
\[
(\nu,t):=(\nu_{M_0}(\theta,\mathcal{I}),type_{M_0}(\theta,\mathcal{I}))
\]
where we recall that the tg-order $\nu_{M_0}(\theta,\mathcal{I})$ stands for the length of the tangency chain $\mathcal{T}g(\theta,\mathcal{I})$ over $M_0$ and the $type_{M_0}(\theta,\mathcal{I})$ stands for the type of this chain at $M_0$ (see section \ref{sec:CIS}). The proof of the Theorem relies on making this pair of invariant dropp (in relation to their lexicographical order). This is obtained through two steps:
\begin{proposition}
Let $(M,M_0,\theta, \mathcal{I},E)$ be a local Log-Canonical (resp. $R$-monomial) $1$-foliated ideal sheaf and suppose that $type_{M_0}(\theta,\mathcal{I})=2$. Then, there exists a sequence of $\theta$-invariant admissible blowings-up of order one:
\[
\begin{tikzpicture}
  \matrix (m) [matrix of math nodes,row sep=3em,column sep=3em,minimum width=1em]
  {\mathcal{S}_1(M,M_0,\theta,\mathcal{I},E): (M_r,\theta_r,\mathcal{I}_r,E_r) & \cdots & (M_0,\theta_0,\mathcal{I}_0,E_0)\\};
  \path[-stealth]
    (m-1-1) edge node [above] {$\sigma_r$} (m-1-2)
    (m-1-2) edge node [above] {$\sigma_1$} (m-1-3);
\end{tikzpicture}
\]
such that:
\begin{itemize}
\item[i]) $\nu_{M_r}(\theta_r,\mathcal{I}_r) \leq \nu_{M_0}(\theta,\mathcal{I})$ and $type_{M_r}(\theta_r,\mathcal{I}_r)=1$;
\item[ii]) The singular distribution $\theta_r$ is Log-Canonical (resp. $R$-monomial);
\item[iii]) $\mathcal{S}_1$ is a chain-preserving blow-up functor (i.e.,it commutes with chain-preserving smooth morphisms).
\end{itemize}
\label{prop:Hi1a}
\end{proposition}
\begin{proposition}
Let $(M,M_0,\theta, \mathcal{I},E)$ be a local Log-Canonical (resp. $R$-monomial) $1$-foliated ideal sheaf and suppose that $type_{M_0}(\theta,\mathcal{I})=1$. Then, there exists a sequence of $\theta$-transverse admissible blowings-up of order one:
\[
\begin{tikzpicture}
  \matrix (m) [matrix of math nodes,row sep=3em,column sep=3em,minimum width=1em]
  {\mathcal{S}_2(M,M_0,\theta,\mathcal{I},E): (M_r,\theta_r,\mathcal{I}_r,E_r) & \cdots & (M_0,\theta_0,\mathcal{I}_0,E_0)\\};
  \path[-stealth]
    (m-1-1) edge node [above] {$\sigma_r$} (m-1-2)
    (m-1-2) edge node [above] {$\sigma_1$} (m-1-3);
\end{tikzpicture}
\]
such that:
\begin{itemize}
\item[i]) $\nu_{M_r}(\theta_r,\mathcal{I}_r) < \nu_{M_0}(\theta,\mathcal{I})$;
\item[ii]) The singular distribution $\theta_r$ is Log-Canonical (resp. $R$-monomial);
\item[iii]) $\mathcal{S}_2$ is a chain-preserving blow-up functor (i.e.,it commutes with chain-preserving smooth morphisms).
\end{itemize}
\label{prop:Hi1b} 
\end{proposition}
%
 These two Propositions are proved in the next two subsections. We note that the proof of the Theorem trivially follows from these Propositions:
%
\begin{proof}
(Theorem \ref{th:Hi1}): Indeed, by a recursive use of Proposition \ref{prop:Hi1a} and \ref{prop:Hi1b}, we obtain a sequence of $\theta$-admissible blowings-up
\[
\begin{tikzpicture}
  \matrix (m) [matrix of math nodes,row sep=3em,column sep=3em,minimum width=1em]
  { (M_r,\theta_r,\mathcal{I}_r,E_r) & \cdots & (M_0,\theta_0,\mathcal{I}_0,E_0)\\};
  \path[-stealth]
    (m-1-1) edge node [above] {$\sigma_r$} (m-1-2)
    (m-1-2) edge node [above] {$\sigma_1$} (m-1-3);
\end{tikzpicture}
\]
such that $\nu(\theta_r,\mathcal{I}_r) = 0$ and $type(\theta_r,\mathcal{I}_r)=1$. In other words, this implies that $\mathcal{I}_r = \mathcal{O}_{M_r}$. Moreover, conditions $[ii]$ and $[iii]$ of the Propositions, imply conditions  $[ii]$ and $[iii]$ of the Theorem $\ref{th:Hi1}$.
\end{proof}
\subsection{Proof of Proposition \ref{prop:Hi1a}}
Consider a $1$-foliated Log-Canonical (resp. $R$-monomial) ideal sheaf $(M,M_0,\theta, \allowbreak \mathcal{I},E)$ such that $type_{M_0}(\theta,\mathcal{I}) = 2$. Let $\nu := \nu_{M_0}(\theta,\mathcal{I})$ and $\mathcal{C}l:=H(\theta,\mathcal{I},\nu)$ (see section \ref{sec:CIS} for the definition of $H(\theta,\mathcal{I},i)$). By Theorem $\ref{th:HironakaS}$, there exists a $\theta$-invariant resolution $\vec{\sigma} = (\sigma_1,...,\sigma_r)$ of $(M,M_0,\theta,\mathcal{C}l,E)$:
\[
\begin{tikzpicture}
  \matrix (m) [matrix of math nodes,row sep=3em,column sep=3em,minimum width=2em]
  {(M_r,\theta_r,\mathcal{C}l_r,E_r) & \cdots & (M_1,\theta_1,\mathcal{C}l_1,E_1) & (M_0,\theta_0,\mathcal{C}l_0,E_0)\\};
  \path[-stealth]
    (m-1-1) edge node [above] {$\sigma_r$} (m-1-2)
    (m-1-2) edge node [above] {$\sigma_2$} (m-1-3)
    (m-1-3) edge node [above] {$\sigma_1$} (m-1-4);
\end{tikzpicture}
\]
\begin{claim}
 The sequence of blowings-up $\vec{\sigma}$ is of order one in relation to $(M_0,\theta_0, \allowbreak \mathcal{I}_0,E_0)$. Furthermore, $\mathcal{C}l_{j} = H(\theta_j,\mathcal{I}_j,\nu)$ for all $j  \leq r$.
 \label{cl:1a}
\end{claim}
\begin{flushleft}
The main step for proving the claim is the following Lemma:
\end{flushleft}
\begin{lemma}
Let $\sigma: (\widetilde{M},\widetilde{\theta},\widetilde{\mathcal{I}},\widetilde{E}) \longrightarrow (M,\theta,\mathcal{I},E)$ be a $\theta$-invariant blowing-up with center contained in $V(\mathcal{C}l)$, where $\mathcal{C}l:=H(\theta,\mathcal{I},\nu)$. Then
\[
H(\widetilde{\theta},\widetilde{\mathcal{I}}, i) = H(\theta,\mathcal{I}, i)^{\ast}  \mathcal{O}(F) 
\]
for every $i  \leq \nu$, where $F$ is the exceptional divisor created by $\sigma$. In particular $\widetilde{\mathcal{C}l} = H(\widetilde{\theta},\widetilde{\mathcal{I}},\nu)$.
\label{lem:25}
\end{lemma}
\begin{proof}
First, notice that, since $H(\theta,\mathcal{I}, i) \subset \mathcal{C}l$ for $i \leq \nu$, the center of blowing-up is also contained in $H(\theta,\mathcal{I},i)$ for all $i\leq \nu$. Furthermore, the center $\mathcal{C}$ is $\theta$-invariant and, by Lemma \ref{lem:thetainv}, the singular distribution $\widetilde{\theta}$ coincides with the total transform $\theta^{\ast}$. Thus, if $\mathcal{J}$ is a coherent ideal sheaf, by Lemma \ref{lem:AlgProp}
\[
\widetilde{\theta} [\mathcal{O}(-F)] \subset \mathcal{O}(-F) \Rightarrow \mathcal{J} \widetilde{\theta}[\mathcal{O}(F)] \subset  \mathcal{J} \mathcal{O}(F)
\]
In particular, this implies that:
\[
\widetilde{\theta}[\mathcal{J}\mathcal{O}(F)] + \mathcal{J}\mathcal{O}(F) = \mathcal{O}(F)(\widetilde{\theta}[\mathcal{J}] +\mathcal{J})
\]
Now, it rests to prove that the following equality:
\[
H(\widetilde{\theta},\widetilde{\mathcal{I}}, i) =  H(\theta,\mathcal{I}, i)^{\ast} \mathcal{O}(F)
\]
is valid for all $i \leq \nu$. Indeed, suppose by induction that the equality is valid for $i < k$ (notice that for $k=0$, the equality is trivial). Since the center of blowing-up is contained in $V(H(\theta,\mathcal{I}))$, by Lemma \ref{lem:AlgProp} we have that:
\[
\begin{aligned}
H(\widetilde{\theta},\widetilde{\mathcal{I}}, k) &= H(\widetilde{\theta},\widetilde{\mathcal{I}}, k-1) + \widetilde{\theta}[H(\widetilde{\theta}, \widetilde{\mathcal{I}},k-1)]\\
&=H(\widetilde{\theta},\widetilde{\mathcal{I}}, k-1) + \theta^{\ast}[ H(\theta,\mathcal{I}, k-1)^{\ast} \mathcal{O}( F)] \\
&=\mathcal{O}(F) \{ H(\theta,\mathcal{I}, k-1)^{\ast} +\theta^{\ast}[H(\theta,\mathcal{I}, k-1)^{\ast}] \} \\
&=\mathcal{O}(F) \{ H(\theta,\mathcal{I}, k-1) +\theta[H(\theta,\mathcal{I}, k-1)] \}^{\ast} \\
&=\mathcal{O}(F) H(\theta,\mathcal{I}, k)^{\ast}
\end{aligned}
\]
which proves the equality and the Lemma.
\end{proof}
We now turn to the Proof of the Claim \ref{cl:1a}:
\begin{proof}[Proof of Claim \ref{cl:1a}]
Suppose by induction that for $i<k$, the sequence of blowings-up $(\sigma_i , ... , \sigma_1)$ is of order one in respect to $(M_0,\theta_0,\mathcal{I}_0,E_0)$ and that $\mathcal{C}l_{i} = H(\theta_i,\mathcal{I}_i,\nu)$. Let us prove the result $i=k$ (including $k=1$). Since the center of $\sigma_k$ is contained in $V(\mathcal{C}l_{k-1})$, by the induction hypotheses it is also contained in $V(\mathcal{I}_{k-1})$, which implies that the  sequence of blowings-up $(\sigma_k , ... , \sigma_1)$ is of order one in respect to $(M_0,\theta_0,\mathcal{I}_0,E_0)$. Furthermore, by Lemma $\ref{lem:25}$ and the induction hypotheses:
\[
H(\theta_k,\mathcal{I}_k,\nu) = H(\theta_{k-1},\mathcal{I}_{k-1},\nu)^{\ast}.\mathcal{O}(F) = [\mathcal{C}l_{k-1}]^{\ast}.\mathcal{O}(F) = \mathcal{C}l_k
\]
which finishes the proof.
\end{proof}
So, if we take the same sequence of blowing-up $\vec{\sigma}$, but in respect to $(M,M_0,\theta, \allowbreak \mathcal{I},E)$, we obtain a $\theta$-invariant sequence of blowings-up of order one:
\[
\begin{tikzpicture}
  \matrix (m) [matrix of math nodes,row sep=3em,column sep=3em,minimum width=2em]
  {(M_r,\theta_r,\mathcal{I}_r,E_r) & \cdots & (M_1,\theta_1,\mathcal{I}_1,E_1) & (M_0,\theta_0,\mathcal{I}_0,E_0)\\};
  \path[-stealth]
    (m-1-1) edge node [above] {$\sigma_r$} (m-1-2)
    (m-1-2) edge node [above] {$\sigma_2$} (m-1-3)
    (m-1-3) edge node [above] {$\sigma_1$} (m-1-4);
\end{tikzpicture}
\]
such that $ H(\theta_r,\mathcal{I}_r,\nu) = \mathcal{C}l_r = \mathcal{O}_{M_r}$, which implies that $\nu_{M_r}(\theta_r,\mathcal{I}_r) \leq  \nu_{M_0}(\theta,\mathcal{I})$ and $type_{M_r}(\theta_r,\mathcal{I}_r)=1$.\\
\\
Now, since all blowings-up are $\theta$-invariant, by Proposition \ref{prop:AdmCenter}, the singular distribution $\theta_r$ is Log-Canonical (resp. $R$-monomial). So, it only rests to prove the Functoriality Statement $[iii]$.\\
\\
To this end, let $\phi: (M,M_0,\allowbreak \theta,\mathcal{I},E_M) \longrightarrow (N,N_0,\omega,\mathcal{J},E_N)$ be a chain-preserving smooth morphism and $\vec{\sigma} = (\sigma_1,...,\sigma_r)$ and $\vec{\tau} = (\tau_1,...,\tau_r)$ be the sequences of blowings-up given in the above algorithm applied to $(M,M_0,\allowbreak \theta,\mathcal{I},E_M)$ and $ (N,N_0,\omega,\mathcal{J}, \allowbreak E_N)$ respectively (the length of the sequence may be chosen to be the same up to isomorphisms). Since $\phi$ is Chain-Preserving, we have that
\[
 H(\theta_0,\mathcal{I}_0,\nu) =  H(\omega_0,\mathcal{J}_0,\nu).\mathcal{O}_{M_0}
\]
Now, since $\vec{\sigma}$ is the sequence of blowing-up given by Theorem $\ref{th:HironakaS}$ that resolves $H(\theta_0,\mathcal{I}_0,\nu)$ and $ \vec{\tau}$ is the sequence of blowing-up given by Theorem $\ref{th:HironakaS}$ that resolves $ H(\omega_0,\mathcal{J}_0,\nu)$, by the functoriality of Theorem $\ref{th:HironakaS}$ the two sequences of blowings-up $\vec{\sigma}$ and $ \vec{\tau}$ commute in respect to smooth morphisms. In particular, for any ideal sheaf $\mathcal{K}$ over $N_{i-1}$:
\[
(\sigma_i)^{\ast}(\mathcal{K} . \mathcal{O}_{M_{i-1}}) = ( \tau_i^{\ast}\mathcal{K}).\mathcal{O}_{M_i}
\]
So, if $F_{M,i}$ is the exceptional divisor of the blowing-up $\sigma_i:M_{i} \longrightarrow M_{i-1}$ and $F_{N,i}$ is the exceptional divisor of the blowing-up $\tau_i:N_{i} \longrightarrow N_{i-1}$, we have that:
\[
\mathcal{O}(-F_{N,i}) . \mathcal{O}_{M_i} = \mathcal{O}(-F_{M,i})
\]
Furthermore the equality $H(\mathcal{J}_i,\omega_i,j) . \mathcal{O}_{M_i} = H(\mathcal{I}_i,\theta_i,j)$ holds for $i \leq r$ and $j \leq \nu$. Indeed, we suppose by induction that $H(\omega_i,\mathcal{J}_i,j) . \mathcal{O}_{M_i} = H(\theta_i,\mathcal{I}_i,j)$ for $i <k$ and any $j \in \mathbb{N}$. Then:
\[
\begin{aligned}
 H(\omega_k,\mathcal{J}_k,j) . \mathcal{O}_{M_k} &= \left[\mathcal{O}(F_{N,k}) \tau_k^{\ast}H(\omega_{k-1},\mathcal{J}_{k-1},j)\right]. \mathcal{O}_{M_k}\\
&= \mathcal{O}(  F_{M,k}) \sigma^{\ast}_{k}H(\theta_{k-1},\mathcal{I}_{k-1},j)\\
&= H(\theta_k,\mathcal{I}_k,j)
\end{aligned}
\]
for any $j \in \mathbb{N}$, which implies that the two sequences of blowings-up $\vec{\sigma}$ and $\vec{\tau}$ commute by chain-preserving smooth morphisms. This finishes the proof.
\subsection{Proof of Proposition \ref{prop:Hi1b}}
Consider a $1$-foliated ideal sheaf $(M,M_0,\theta,\mathcal{I},E)$ such that $type_{M_0}(\theta,\mathcal{I}) = 1$. Let $\nu := \nu_{M_0}(\theta,\mathcal{I})$ and $\mathcal{M}tg:=H(\theta,\mathcal{I},\nu-1)$. By Theorem $\ref{th:Hironaka}$, there exists a resolution $\vec{\sigma} = (\sigma_1,...,\sigma_r)$ of $(M,M_0,\theta,\mathcal{M}tg(\mathcal{I}),E)$:
\[
\begin{tikzpicture}
  \matrix (m) [matrix of math nodes,row sep=3em,column sep=3em,minimum width=2em]
  {(M_r,\theta_r,\mathcal{M}tg_r,E_r) & \cdots & (M_0,\theta_0,\mathcal{M}tg_0,E_0)\\};
  \path[-stealth]
    (m-1-1) edge node [above] {$\sigma_r$} (m-1-2)
    (m-1-2) edge node [above] {$\sigma_1$} (m-1-3);
\end{tikzpicture}
\]
\begin{claim}
 The sequence of blowings-up $\vec{\sigma}$ is $\theta$-transverse.
 \label{cl:1b}
\end{claim}
\begin{proof}
Suppose by strong induction that, for $i < k$:
\begin{itemize}
\item[a]) The sequence $(\sigma_1,...,\sigma_i)$ of blowing-up is $\theta$-transverse;
\item[b]) For any point $p$ in $ V( \mathcal{M}tg_i )$, there exists a coherent coordinate system $(x,\mathbf{y}) = (x,y_1,...,y_{m-1})$ centered at $p$ such that $x \in \mathcal{M}tg_i . \mathcal{O}_p$ and $\partial x$ generates $\theta_{i} . \mathcal{O}_p$.
\end{itemize}
We prove the result for $k$:\\
\\
\textbf{Base Step}: We start proving the result for $k=0$. In this case, since $type_{M_0}(\theta,\mathcal{I}) = 1$, we know that $\theta(\mathcal{M}tg_0) = \mathcal{O}_{M_0}$, which implies that $\theta.\mathcal{O}_p$ is generated by a regular vector-field  $\partial$. By the flow-box Theorem, there exists a coordinate system $(x,\mathbf{y}) = (x,y_1,...,y_{m-1})$ where $ \partial = \partial_x$. Furthermore, there exists a function $g \in \mathcal{M}tg_0$ such that $\partial_x(g)$ is a unit. Taking $\bar{x}=g$ and $\bar{\mathbf{y}}=\mathbf{y}$, we conclude the result. \\
\\
\textbf{Induction Step:} Let us prove for $k>0$. Fix a point $q$ in $ V( \mathcal{M}tg_{k-1} )$. Since the center $\mathcal{C}_k$ of the blowing-up $\sigma_k: M_k \to M_{k-1} $ is contained in $ V( \mathcal{M}tg_{k-1} )$, by the induction hypotheses $[b]$ it is also totally transverse to $\theta$ at $q$. Since the choice of $q$ was arbitrary, the sequence of blowings-up $(\sigma_1,...,\sigma_k)$ is $\theta$-transverse.\\
\\
Consider now a point $p$ in $V(\mathcal{M}tg_{k})$ and let $q=\sigma_k(p)$. If $\sigma_k$ is a local isomorphism over $p$, the condition $[b]$ is trivially satisfied at $p$. So, let us assume that $p \in F_k$. By the induction hypotheses $[b]$, there exists a coherent coordinate system $(x,\mathbf{y}) = (x,y_1,...,y_{m-1})$ of $\mathcal{O}_q$ such that $x \in \mathcal{M}tg_{k-1} . \mathcal{O}_q$ and $\partial x$ generates $\theta_{k-1}.\mathcal{O}_q$. Since $\mathcal{C} \subset V(\mathcal{M}tg_{k-1})$, without loss of generality $\mathcal{I}_{\mathcal{C}}.\mathcal{O}_p = (x,y_1,...,y_t)$ and $q$ is the origin of the $y_1$-chart. It is now easy to compute the transforms of the blowing-up at $q$ and see that the induction hypotheses $[b]$ is valid at $p$.
\end{proof}
\begin{flushleft}
 By claim $\ref{cl:1b}$ and Lemma $\ref{lem:thetatr}$, we deduce that:
\end{flushleft}
\begin{equation}
\theta_{k+1} =  \mathcal{O}(-F_k) \sigma_k^{\ast}(\theta_k)
\label{eq:te1}
\end{equation}
and, since the center is totally transverse:
\begin{equation}
\theta_{k+1}[\mathcal{O}(-F_k)] \subset  \mathcal{O}(-F_k)
\label{eq:te2}
\end{equation}
In order to simplify the notation, let us define $(i \sigma_k) = \sigma_{i+1}\circ ...  \circ \sigma_{k}$ for $i <  k$, $(k \sigma_k) = id$ and $\boldsymbol{\sigma_k} = \sigma_1 \circ ... \circ \sigma_k$. We also introduce:
\[
\mathcal{K}_k(\alpha) = \prod^{k-1}_{i=1} [(i\sigma_{k-1})^{\ast}\mathcal{O}( \alpha F_i)]
\]
Using the equation $\eqref{eq:te1}$ recursively, we get that:
\begin{equation}
\theta_k= \mathcal{K}_k(-1) \boldsymbol{\sigma_k}^{\ast} \theta
\label{eq:1a}
\end{equation}
Using the equation $\eqref{eq:te2}$ recursively, we get that:
\begin{equation}
\theta_k(\mathcal{K}_k(\alpha)) \subset \mathcal{K}_k(\alpha)
\label{eq:1b}
\end{equation}
Furthermore, given an ideal sheaf $\mathcal{J}$, equation $\eqref{eq:1b}$ implies that:
\begin{equation}
\theta_k[\mathcal{K}_k(\alpha) \mathcal{J}] + \mathcal{K}_k(\alpha) \mathcal{J} = \mathcal{K}_k(\alpha)( \mathcal{J} +\theta_k [\mathcal{J}])
\label{eq:2}
\end{equation} 
\begin{claim}
 The sequence of blowing up $\vec{\sigma}$ is of order one in respect to $(M_0,\theta_0, \allowbreak \mathcal{I}_0,E_0)$. Furthermore, for all $j \leq \nu$:
\begin{align}
H(\theta_r,\mathcal{I}_r,j) &= \mathcal{K}_r(1) . \sum^j_{i=0} \mathcal{K}_r(-i) \boldsymbol{\sigma}^{\ast} H(\theta_0,\mathcal{I}_0,i)
\label{eq:3}
\end{align}
\label{cl:2b}
\end{claim}
\begin{proof}
Suppose by strong induction that, for $k<k_0$:
\begin{itemize}
\item[a]) The sequence of blowing up $(\sigma_k,...,\sigma_1)$ is of order one in respect to $(M_0,\theta_0,\mathcal{I}_0,E_0)$;
\item[b]) Equation $\eqref{eq:3}$ is valid for $k<k_0$ instead of $r$.
\end{itemize}
We prove the result for $k$. But notice that the base step $k_0=0$ is trivial, so we only need to consider the inductive step:\\
\\
\textbf{Induction Step:} Using the induction hypotheses $[b]$ we deduce that:
\[
\begin{aligned}
H(\theta_{k_0-1},\mathcal{I}_{k_0-1},j) &= \mathcal{K}_{k_0-1}(1) . \sum^{j}_{i=0} \mathcal{K}_{k_0-1}(-i)   \boldsymbol{\sigma_{k_0-1}}^{\ast} H(\theta_0,\mathcal{I}_0,i) \\
&\subset \mathcal{K}_{k_0-1}(1). \sum^{j}_{i=0}   \boldsymbol{\sigma_{k_0-1}}^{\ast} H(\theta_0,\mathcal{I}_0,i) \\
&= \mathcal{K}_{k_0-1}(1) \boldsymbol{\sigma_{k_0-1}}^{\ast} H(\theta_0,\mathcal{I}_0,j)
\end{aligned}
\]
In particular $H(\theta_{k_0-1},\mathcal{I}_{k_0-1},\nu-1) \subset \mathcal{M}tg_{k_0-1}$, which implies that $\mathcal{C}_{k_0} \subset V(H(\theta_{k_0-1},\mathcal{I}_{k_0-1},\nu-1))$. So the sequence of blowing up $(\sigma_{k_0},...,\sigma_1)$ is of order one in respect to $(M_0,\theta_0,\mathcal{I}_0,E_0)$.\\
\\
We now need to verify the induction hypotheses $[b]$ for $k=k_0$. We do this by induction on $j \leq \nu$. Indeed, the hypothesis $[b]$ is clearly true for $j=0$, so we can assume by strong induction that it is also true for $j<j_0$. By equation $\eqref{eq:2}$:
\[
\begin{aligned}
H(\theta_{k_0},\mathcal{I}_{k_0},j_0&) = H(\theta_{k_0},\mathcal{I}_{k_0},j_0-1) + \theta_{k_0}[ H(\theta_{k_0},\mathcal{I}_{k_0},j_0-1)]\\
&=H(\theta_{k_0},\mathcal{I}_{k_0},j_0-1) + \theta_{k_0}[ \mathcal{K}_{k_0}(1) . \sum^{j_0-1}_{i=0} \mathcal{K}_{k_0}(-i)  \boldsymbol{\sigma_{k_0}}^{\ast} H(\theta_0,\mathcal{I}_0,i)]\\
&=H(\theta_{k_0},\mathcal{I}_{k_0},j_0-1) + \mathcal{K}_{k_0}(1) . \sum^{j_0-1}_{i=0} \mathcal{K}_{k_0}(-i) \theta_{k_0}[ \boldsymbol{\sigma}_{k_0}^{\ast} H(\theta_0,\mathcal{I}_0,i)]
\end{aligned}
\]
Now, using equation $(\ref{eq:1a})$ and Lemma \ref{lem:AlgProp}, we can continue the deduction of $H(\theta_{k_0},\mathcal{I}_{k_0},j_0)$:
\[
\begin{aligned}
&= \mathcal{K}_{k_0}(1) \sum^{j_0-1}_{i=0} \mathcal{K}_{k_0}(-i)  \left[ \boldsymbol{\sigma_{k_0}}^{\ast} H(\theta_0,\mathcal{I}_0,i) +   \mathcal{K}_{k_0}(-1)          \boldsymbol{\sigma_{k_0}}^{\ast}\theta_{0}[ \boldsymbol{\sigma_{k_0}}^{\ast} H(\theta_0,\mathcal{I}_0,i) \right]\\
&=\mathcal{K}_{k_0}(1) \sum^{j_0-1}_{i=0} \mathcal{K}_{k_0}(-i)  \left[ \boldsymbol{\sigma_{k_0}}^{\ast} H(\theta_0,\mathcal{I}_0,i) +   \mathcal{K}_{k_0}(-1)          \boldsymbol{\sigma_{k_0}}^{\ast} \{ \theta_{0}[ H(\theta_0,\mathcal{I}_0,i) \right]\}]\\
&=\mathcal{K}_{k_0}(1) \sum^{j_0-1}_{i=0} \mathcal{K}_{k_0}(-i)  \left[ \boldsymbol{\sigma_{k_0}}^{\ast} H(\theta_0,\mathcal{I}_0,i) +   \mathcal{K}_{k_0}(-1) \boldsymbol{\sigma_{k_0}}^{\ast} H(\theta_0,\mathcal{I}_0,i+1)\right] \\
&=\mathcal{K}_{k_0}(1). \sum^{j_0}_{i=0} \mathcal{K}_{k_0}(-i) \boldsymbol{\sigma_{k_0}}^{\ast} H(\theta_0,\mathcal{I}_0,i)
\end{aligned}
\]
So the formula is proved.
\end{proof}
So, if we take the same sequence of blowing-up $\vec{\sigma}$, but in respect to $(M,M_0,\theta,\mathcal{I},E)$, we obtain a $\theta$-transverse sequence of blowings-up of order one:
\[
\begin{tikzpicture}
  \matrix (m) [matrix of math nodes,row sep=3em,column sep=3em,minimum width=2em]
  {(M_r,\theta_r,\mathcal{I}_r,E_r) & \cdots & (M_1,\theta_1,\mathcal{I}_1,E_1) & (M_0,\theta_0,\mathcal{I}_0,E_0)\\};
  \path[-stealth]
    (m-1-1) edge node [above] {$\sigma_r$} (m-1-2)
    (m-1-2) edge node [above] {$\sigma_2$} (m-1-3)
    (m-1-3) edge node [above] {$\sigma_1$} (m-1-4);
\end{tikzpicture}
\]
By Claim \ref{cl:2b} and the fact that $\mathcal{M}tg_r =  \mathcal{O}_{M_{r}}$, we conclude that $\boldsymbol{\sigma}^{\ast} H(\theta_0,\mathcal{I}_0,\nu-1) = \mathcal{K}_r(1)$. So:
\[
\begin{aligned}
H(\theta_r,\mathcal{I}_r,\nu-1) &= \mathcal{K}_r(1) . \sum^{\nu-1}_{i=0} \mathcal{K}_r(-i) \boldsymbol{\sigma}^{\ast} H(\theta_0,\mathcal{I}_0,i)\\
&= \mathcal{K}_r(1) . \sum^{\nu-2}_{i=0} \mathcal{K}_r(-i)  \boldsymbol{\sigma}^{\ast} H(\theta_0,\mathcal{I}_0,i) + \mathcal{K}_r(- (\nu-2))\\
&= H(\theta_r,\mathcal{I}_r,\nu-2) + \mathcal{K}_r(- (\nu-2))
\end{aligned}
\]
which implies that:
\[
\begin{aligned}
H(\theta_r,\mathcal{I}_r,\nu) &= H(\theta_r,\mathcal{I}_r,\nu-1) + \theta_r[\mathcal{K}_r(- (\nu-2))] \\
&\subset H(\theta_r,\mathcal{I}_r,\nu-1) + \mathcal{K}_r(- (\nu-2))\\
&= H(\theta_r,\mathcal{I}_r,\nu-1)
\end{aligned}
\]
which proves that $\nu_{M_r}(\theta_r,\mathcal{I}_r)$ is strictly smaller then $\nu$. Now, since all blowings-up are $\theta$-transverse, by Proposition \ref{prop:AdmCenter}, the singular distribution $\theta_r$ is Log-Canonical (resp. $R$-monomial). So, it only rests to prove the Functoriality Statement $[iii]$.\\
\\
To this end, let $\phi: (M,M_0,\allowbreak \theta,\mathcal{I},E_M) \longrightarrow (N,N_0,\omega,\mathcal{J},E_N)$ be a chain-preserving smooth morphism and $\vec{\sigma} = (\sigma_1,...,\sigma_r)$ and $\vec{\tau} = (\tau_1,...,\tau_r)$ be the sequences of blowings-up given in the above algorithm applied $(M,M_0,\allowbreak \theta,\mathcal{I},E_M)$ and $ (N,N_0,\omega,\mathcal{J}, \allowbreak E_N)$ respectively (the length of the sequence may be chosen to be the same up to isomorphisms). Since $\phi$ is Chain-Preserving, we have that
\[
 H(\theta_0,\mathcal{I}_0,\nu-1) =  H(\omega_0,\mathcal{J}_0,\nu-1).\mathcal{O}_{M_0}
\]
Now, since $\vec{\sigma}$ is the sequence of blowing-up given by Theorem $\ref{th:HironakaS}$ that resolves $H(\theta_0,\mathcal{I}_0,\nu-1)$ and $ \vec{\tau}$ is the sequence of blowing-up given by Theorem $\ref{th:HironakaS}$ that resolves $ H(\omega_0,\mathcal{J}_0,\nu-1)$, by the functoriality of Theorem $\ref{th:HironakaS}$ the two sequences of blowings-up $\vec{\sigma}$ and $ \vec{\tau}$ commute by smooth morphisms. In particular, for any ideal sheaf $\mathcal{K}$ over $N_{i-1}$:
\[
(\sigma_i)^{\ast}(\mathcal{K} . \mathcal{O}_{M_{i-1}}) = ( \tau_i^{\ast}\mathcal{K}).\mathcal{O}_{M_i}
\]
So, if $F_{M,i}$ is the exceptional divisor of the blowing-up $\sigma_i:M_{i} \longrightarrow M_{i-1}$ and $F_{N,i}$ is the exceptional divisor of the blowing-up $\tau_i:N_{i} \longrightarrow N_{i-1}$, we have that:
\[
\mathcal{O}(-F_{N,i}) . \mathcal{O}_{M_i} = \mathcal{O}(-F_{M,i})
\]
Furthermore, define $\mathcal{K}_{M,k}(-\alpha)$ and $\mathcal{K}_{N,k}(-\alpha)$ in the obvious way. We have that:
\[
\mathcal{K}_{N,i}(-\alpha) . \mathcal{O}_{M_i} = \mathcal{K}_{M,i}(-\alpha)
\]
Now, the equality $H(\mathcal{J}_i,\omega_i,j) . \mathcal{O}_{M_i} = H(\mathcal{I}_i,\theta_i,j)$ holds for $i \leq r$ and $j \leq \nu$. Indeed, we suppose by induction that $H(\omega_i,\mathcal{J}_i,j) . \mathcal{O}_{M_i} = H(\theta_i,\mathcal{I}_i,j)$ for $i <k_0$ and any $j \in \mathbb{N}$. Then:
\[
\begin{aligned}
H(\omega_{k_0},\mathcal{J}_{k_0},j) . \mathcal{O}_{M_{k_0}} &= ( \mathcal{K}_{N,k_0}(1) \sum^j_{i=0} \mathcal{K}_{N,k_0}(-i) \boldsymbol{\tau_{k_0}}^{\ast}H(\omega_0,\mathcal{J}_0,i) ) . \mathcal{O}_{M_{k_0}}\\
&=\mathcal{K}_{M,k_0}(1) \sum^j_{i=0} \mathcal{K}_{M,k_0}(-i) \boldsymbol{\sigma_{k_0}}^{\ast}H(\theta_0,\mathcal{I}_0,i) \\
&= H(\theta_{k_0},\mathcal{I}_{k_0},j)
\end{aligned}
\]
for any $j \in \mathbb{N}$, which implies that the two sequences of blowings-up $\vec{\sigma}$ and $\vec{\tau}$ commute by chain-preserving smooth morphisms. This finishes the proof of the Proposition.
\section*{Acknowledgments}
I would like to express my gratitude to my adviser, Professor Daniel Panazzolo, for the useful discussions, suggestions and revision of the manuscript. This work owns a great deal to his influence. I would also like to express my gratitude to Professor Vincent Grandjean, Orlando Villamayor and Santiago Encinas for the useful discussions on the subject and to Nico Meffe for his generous help with the english grammar. This work was funded by the Universit\'{e} de Haute Alsace.
\bibliographystyle{alpha}

\begin{thebibliography}{999}

\bibitem[ADK]{Abr} Abramovich, D.; Denef, J.; Karu, K. Weak toroidalization over non-closed fields, Manuscripta Math. 142 (2013), no. 1-2, 257-271.
\bibitem[BaBo]{BaumBott} Baum, Paul; Bott, Raoul Singularities of holomorphic foliations. J. Differential Geometry 7 (1972), 279-342.
\bibitem[Bel1]{BeloT} Belotto, A. (2013) Resolution of singularities in foliated spaces. PhD thesis. Universit\'{e} de Haute-Alsace: France.
\bibitem[Bel2]{Bel2} Belotto, A. Local monomialization of a system of first integrals (2014) arXiv:1411.5333
\bibitem[BiMil1]{BM} Bierstone, Edward; Milman, Pierre D. Canonical resolution in characteristic zero by blowing up the maximum strata of a local invariant. Invent. Math. 128 (1997), no. 2, 207-302.
\bibitem[BiMil2]{BM2} Bierstone, Edward; Milman, Pierre D. Functoriality in resolution of singularities. Publ. Res. Inst. Math. Sci. 44 (2008), no. 2, 609-639.
\bibitem[Cu1]{Cut} Cutkosky, Steven Dale. Local monomialization and factorization of morphisms. Ast\'{e}risque No. 260 (1999), vi+143 pp.
\bibitem[Cu2]{Cut2} Cutkosky, Steven Dale Local monomialization of transcendental extensions. Ann. Inst. Fourier (Grenoble) 55 (2005), no. 5, 1517-1586.
\bibitem[Cu3]{Cut3} Cutkosky, Steven Dale Monomialization of morphisms from 3-folds to surfaces. Lecture Notes in Mathematics, 1786. Springer-Verlag, Berlin, 2002.
\bibitem[DR]{RZ} Denkowska, Zofia; Roussarie, Robert A method of resolution for analytic two-dimensional vector field families. Bol. Soc. Brasil. Mat. (N.S.) 22 (1991), no. 1, 93-126.
\bibitem[ENV]{Vil1} Santiago Encinas, Augusto Nobile, and Orlando Villamayor, On algorithmic equi-resolution and stratification of Hilbert schemes , Proc. London Math. Soc. (3) 86 (2003), no. 3, 607-648.
\bibitem[Hi]{Hir} Hironaka, Heisuke Resolution of singularities of an algebraic variety over a field of characteristic zero. I, II. Ann. of Math. (2) 79 (1964), 109-203; ibid. (2) 79 1964 205-326.
\bibitem[Ho]{Hor} L. H\"{o}rmander, An introduction to complex analysis in several variables, North-Holland Publishing Co.L, 1973.
\bibitem[Ki]{King} King, Henry C. Resolving singularities of maps. (English summary) Real algebraic geometry and topology (East Lansing, MI, 1993), 135-154, Contemp. Math., 182, Amer. Math. Soc., Providence, RI, 1995. 
\bibitem[Ko]{Ko} Koll\`{a}r, J\'{a}nos Lectures on resolution of singularities. Annals of Mathematics Studies, 166. Princeton University Press, Princeton, NJ, 2007.
\bibitem[Mc]{Mc} M. McQuillan Canonical models of foliations, Pure and applied mathematics quarterly, Vol. 4(3), 2008, pp. 877-1012.
\bibitem[McP]{McP} M. McQuillan, D. Panazzolo Almost \'{E}tale resolution of foliations. Journal of Differential Geometry 95, no. 2 (2013), 279-319.
\bibitem[V1]{Vil2} Villamayor U., O. Resolution in families. Math. Ann. 309 (1997), no. 1, 1-19.
\bibitem[V2]{Vil4} Villamayor, Orlando Constructiveness of Hironaka's resolution. Ann. Sci. \'{E}cole Norm. Sup. (4) 22 (1989), no. 1, 1-32.
\bibitem[W1]{Wo} Wlodarczyk, Jaroslaw. Resolution of singularities of analytic spaces. (English summary) Proceedings of G\"{o}kova Geometry-Topology Conference 2008, 31-63, G\"{o}kova Geometry/Topology Conference (GGT), G\"{o}kova, 2009.
\bibitem[W2]{Wo2} Wlodarczyk, Jaroslaw Simple Hironaka resolution in characteristic zero. J. Amer. Math. Soc. 18 (2005), no. 4, 779-822.
\end{thebibliography}

\end{document}